\documentclass[12pt]{amsart}
\usepackage[T1,T2A]{fontenc}
\usepackage[utf8]{inputenc}
\usepackage[english]{babel}

\makeatletter
\def\@settitle{\begin{center}
\large
\baselineskip14\p@\relax
\bfseries
\@title
\end{center}%
}
\makeatother

\makeatletter
\providecommand*{\approxident}{\mathrel{\mathpalette\@approxident\sim}} 
\newcommand*{\@approxident}[2]{\sbox0{$#1\vcenter{}$}
\sbox2{$\m@th#1\equiv$}\dimen2=\dimexpr\ht2 - \ht0\relax
\sbox4{$\m@th#1\sim$}\dimen4=\dimexpr\ht4-\ht0\relax
\dimen0=\dimexpr-\ht4-\dp4+\dimen2\relax
\vcenter{\offinterlineskip \copy4 \kern\dimen0\copy4\kern\dimen0\copy4\ifdim\dp4=\z@
\kern\dimexpr -\ht0 + \dimen4\relax \fi}} 
\makeatother

\usepackage[dvipsnames]{xcolor}
\usepackage{graphicx}
\usepackage[normalem]{ulem}

\usepackage{tikz}

\usepackage{amsmath}
\usepackage{float}
\usepackage{amsthm}
\usepackage{tikz}
\usepackage{hyperref}

\usepackage{tabularx,multirow,makecell}
\newcolumntype{Y}{>{\centering\arraybackslash}X}
\usepackage{setspace}
\usepackage{ltablex}

\newcounter{NN}\numberwithin{NN}{section}
\renewcommand{\theNN}{{\rm \arabic{NN}${}^o$}}
\def\nr{\refstepcounter{NN}{\theNN}}

\makeatletter
\@addtoreset{equation}{subsection}
\makeatother

\renewcommand\labelenumi{\rm (\roman{enumi})}
\renewcommand\theenumi{\rm (\roman{enumi})}

\usepackage{mathrsfs, amssymb, array,amsfonts, amsmath,upgreek,color}
\usepackage[matrix,arrow,curve]{xy}
\usepackage{hyperref}

\swapnumbers
\theoremstyle{plain}
\newtheorem{theorem}[subsection]{Theorem}
\newtheorem{stheorem}[equation]{Theorem}
\newtheorem{lemma}[subsection]{Lemma}

\newtheorem{sproposition}[equation]{Proposition}
\newtheorem{proposition}[subsection]{Proposition}

\newtheorem{scorollary}[equation]{Corollary}
\newtheorem{slemma}[equation]{Lemma}

\theoremstyle{definition}
\newtheorem{definition}[subsection]{Definition}

\newtheorem{remark}[subsection]{Remark}
\newtheorem{sremark}[equation]{Remark}

\newtheorem{ass}[subsection]{Assumption}

\newcommand{\NE}{\overline{\operatorname{NE}}}
\newcommand{\Eff}{{\operatorname{Eff}}}
\newcommand{\Nef}{{\operatorname{Nef}}}

\newcommand{\p}{\mathrm{p}_{\mathrm{a}}}

\newcommand{\dd}{\mathrm{d}}
\newcommand{\cc}{\mathrm{c}}

\newcommand{\red}{\mathrm{red}}
\newcommand{\Sing}{\operatorname{Sing}}
\newcommand{\Gr}{\operatorname{Gr}}

\newcommand{\Pic}{\operatorname{Pic}}
\newcommand{\Eu}{\operatorname{Eu}}
\newcommand{\Cone}{\operatorname{Cone}}
\newcommand{\J}{\operatorname{J}}
\newcommand{\JG}{\operatorname{J}_{\mathrm{G}}}
\newcommand{\Cl}{\operatorname{Cl}}

\newcommand{\rk}{\operatorname{rk}}
\newcommand{\corank}{\operatorname{corank}}

\newcommand{\PP}{\mathbb{P}}
\newcommand{\RR}{\mathbb{R}}
\newcommand{\QQ}{\mathbb{Q}}
\newcommand{\ZZ}{\mathbb{Z}}

\newcommand{\FF}{\mathbb{F}}

\newcommand{\NNN}{{\mathscr{N}}}
\newcommand{\OOO}{{\mathscr{O}}}
\newcommand{\EEE}{{\mathscr{E}}}
\newcommand{\LLL}{{\mathscr{L}}}

\newcommand{\type}[1]{$\mathrm{#1}$}

\newcommand{\xref}[1]{\textup{\ref{#1}}}

\newcommand{\g}{{\mathrm{g}}}
\newcommand{\bil}[2]{\langle #1,\, #2 \rangle}

\begin{document}
\author{Yuri Prokhorov}
\thanks{
The research has been funded within the framework of the HSE University Basic Research Program.
}
\address{
\newline
Steklov Mathematical Institute of Russian Academy of Sciences, Moscow, 
Russia
\newline
Faculty of Mathematics,
Moscow State University, Moscow, Russia
\newline
AG Laboratory, HSE University, Moscow, 
Russia
}
\email{prokhoro@mi-ras.ru}
\title{Rationality of Fano threefolds with 
terminal Gorenstein singularities, II}
\maketitle

\begin{abstract}
We classify nonrational Fano threefolds $X$ with terminal Gorenstein singularities 
such that $\mathrm{\rk}\, \mathrm{\Pic}(X)=1$, $(-K_X)^3\ge 8$, and $\mathrm{\rk}\, \mathrm{\Cl}(X)\le 2$.
\end{abstract}

\section{Introduction}
Historically, rationality problem for Fano varieties is 
one of the most important and difficult problems in the birational algebraic geometry.
It is mostly solved for smooth Fano threefolds (see \cite[Ch.~8-9]{IP99} and references therein) but it is still widely open for singular ones. 
In this paper we discuss
rationality of Fano threefolds with terminal Gorenstein singularities.
The main reason for studying such objects is 
that they arise
naturally in the minimal model program, in classification of finite subgroups of the space Cremona group \cite{P:G-MMP}, and in many other instances.

This work is a sequel to \cite{P:ratF-1}.
As in \cite{P:ratF-1}, we concentrate on the class of Fano threefolds of Picard number 1 and genus $g\ge 5$
(a similar classification in the case $g\le 4$ is expected to be much longer, cf. \cite{Kaloghiros2012}).
For each value of genus $g=g(X)$ and Fano index $\iota(X)$ these varieties form an irreducible family.
This follows from the result \cite{Namikawa:Fano} and explicit description of smooth members 
of the families \cite[\S\S~3.3, 4.1, 5.1]{IP99}.

A natural invariant of a singular Fano threefold is the rank of the Weil divisor class group $\Cl(X)$. 
The case $\rk\Cl(X)=1$ was considered in \cite{P:ratF-1}
(for the notation we refer to~\ref{r-not}):
\begin{theorem}
\label{th:rho=1}
Let $X$ be a Fano threefold with terminal $\QQ$-factorial singularities and 
$\uprho(X)=1$.
If $X$ is not rational, then for $X$ one of the following possibilities holds:
\begin{enumerate}
\item
\label{fact:iota=2}
$\iota(X)=2$, $\dd(X)\le3$, and if $\dd(X)=3$, then $X$ is nonsingular;

\item
\label{fact:iota=1}
$\iota(X)=1$, $\g(X)\le 6$;

\item
\label{fact:X14}
$\iota(X)=1$, $X=X_{14}\subset\PP^9$ is a smooth threefold of genus $8$.
\end{enumerate}
Conversely, a general variety in the above families is nonrational.
Moreover, any smooth variety from \xref{fact:iota=2} and \xref{fact:X14},
and from \xref{fact:iota=1} with $\g(X)\in \{2,\, 3,\, 5\}$
is nonrational.
\end{theorem}

The first our result concerns with the case $\rk\Cl(X)=2$.

\begin{theorem}
\label{th:rho=2}
Let $X$ be a Fano threefold with terminal Gorenstein singularities
such that $\uprho(X)=1$ and $\rk\Cl(X)=2$. Assume that $X$ is not rational and 
$g:=\g(X)\ge 5$. 
Then there is the following Sarkisov link:
\begin{equation}
\label{eq:diagram}
\vcenter{
\xymatrix{
&\hat{X}\ar[dr]^\tau \ar@{-->}[rr]^{\chi}\ar[dl]_f&&\hat{X}'\ar[dr]^{f'} 
\ar[dl]_{\tau'}
\\
Y&&X&&Y'
}}
\end{equation}
where $\tau$ and $\tau'$ are two small $\QQ$-factorializations and $\chi$ is a flop.
The maps $f$ and $f'$ are described in the following table.
\setlength{\extrarowheight}{1pt}
\newcommand{\heading}[1]{\multicolumn{1}{c}{#1}}
\newcommand{\headingl}[1]{\multicolumn{1}{c}{#1}}
\par\medskip\noindent
\begin{tabularx}{\linewidth}{c|c|X|X|c} 
\hline
& $g$ & $f$ & $f'$ & $\#$
\\\hline 
\multicolumn{5}{c}{symmetric links}
\\\hline 
\nr\label{dP-dP}& $5$ & 
del Pezzo fibration of degree $4$ over $\PP^1$
& del Pezzo fibration of degree $4$ over $\PP^1$& $4$
\\\hline 
\nr\label{blowup-V14}&$5$ & 
blowup of a conic on a smooth threefold $Y_{14}\subset\PP^9$ 
& 
blowup of a conic on a smooth threefold $Y_{14}'\subset\PP^9$ & $10$
\\\hline 
\nr\label{blowup-W2-point}&$5$ & 
blowup of a smooth point on a del Pezzo threefold $Y_2$ 
& 
blowup of a smooth point on $Y_2$ &$12$
\\\hline 
\nr\label{conic-conic}& $8$ & 
conic bundle over $\PP^2$ with discriminant curve of degree $5$& 
conic bundle over $\PP^2$ with discriminant curve of degree $5$
&$1$
\\\hline 
\nr\label{blowup-cubic-point}&$9$ & 
blowup of a point on a smooth cubic $Y_3\subset\PP^4$ 
& 
blowup of a point on $Y_3\subset\PP^4$ &$6$
\\\hline 
\multicolumn{5}{c}{non-symmetric links}
\\\hline 
\nr\label{g=5:trig}&$5$ & 
del Pezzo fibration of degree $3$ over $\PP^1$
& 
conic bundle over $\PP^2$ with discriminant curve of degree $7$ 
&$1$
\\\hline 
\nr\label{g=5:conic-G}& 
$5$ & 
blowup of a \type{cA_1}-point on locally factorial threefold 
$Y_{10}^s\subset\PP^{7}$ 
&
conic bundle over $\PP^2$ with discriminant curve of degree $6$
&$6$
\\\hline 
\nr\label{dP-conic}& $6$ & 
del Pezzo fibration of degree $4$ over $\PP^1$
& 
conic bundle over $\PP^2$ with discriminant curve of degree $6$
& $2$
\\\hline 
\nr\label{conic-blowup}& $6$ & 
blowup of a line on a smooth threefold $Y_{14}\subset\PP^{9}$ 
& 
conic bundle over $\PP^2$ with discriminant curve of degree $5$ 
&$6$  
\\\hline 
\nr\label{g=6:nonG}&$6$ & 
blowup of a rational twisted cubic curve on a smooth 
cubic $Y_3\subset\PP^4$ 
& blowup of a point $\frac12(1,1,1)$ on $Y_{\frac{21}2}$&$6$
\\\hline 
\nr\label{g=6:trig}&$6$ & blowup of a line on a del Pezzo threefold $Y_2$ & 
del Pezzo fibration of degree $3$ &$1$
\\\hline 
\nr\label{dP-blowup}& $8$ & 
blowup of a conic on a smooth cubic $Y_3\subset\PP^4$
& 
del Pezzo fibration of degree $4$ over $\PP^1$
&$1$
\\\hline 
\end{tabularx}
\noindent
where $\#$ is the number of singular points on $X$ under the assumption that the variety is general in the corresponding family. 

\noindent
All the varieties described in 
\ref{blowup-V14}, 
\ref{conic-conic},
\ref{blowup-cubic-point},
\ref{conic-blowup}, 
\ref{g=6:nonG}, 
\ref{dP-blowup}
are not rational.
General varieties described in 
\ref{dP-dP}, 
\ref{blowup-W2-point}, 
\ref{g=5:trig},
\ref{g=5:conic-G}, \ref{dP-conic}, \ref{g=6:trig}
are not rational.
\end{theorem}
Here we say that a link is \emph{symmetric} if the numerical invariants on the left and the right hand sides
of \eqref{eq:diagram}
are the same. However, in the cases~\ref{blowup-W2-point} and~\ref{blowup-cubic-point} we can say more:
the links are symmetric in the strongest sense, i.e. there is an involution that swaps 
$Y$ and $Y'$
(see Remarks~\ref{rem:V1}\ref{rem:V2a} and~\ref{rem:V2}\ref{rem:V2a}).

The variety $Y=Y_{10}^s\subset\PP^{7}$ from~\ref{g=5:conic-G} is locally factorial Fano threefold
with $\uprho(Y_{10}^s)=1$ and $g(Y_{10}^s)=6$ having a singular point of type \type{cA_1}; a general variety from this family 
was considered in \cite{Debarre-Iliev-Manivel-2011}.
The variety $Y'=Y_{\frac{21}2}$ from~\ref{g=6:nonG} is a non-Gorenstein Fano threefold with $(-K_{Y'})^3=\frac{21}2$
whose singular locus consists of a unique point of type $\frac12(1,1,1)$ \cite[No.~29375]{GRD}.
%

Note that in many cases the link \eqref{eq:diagram} was explicitly described earlier, see \cite[Ch.~4]{IP99}, 
\cite[Sect.~7]{Jahnke-Peternell-Radloff-II}, \cite{Takeuchi-2009}, \cite{P:factorial-Fano:e}.
However, typically the authors considered the link \eqref{eq:diagram}
with smooth varieties $\hat X$ and $\hat X'$.
For convenience of the reader we reproduce the corresponding constructions
in full generality.
Note also that the variety~\ref{conic-conic} is completely new: it was erroneously omitted in \cite[7.5]{Jahnke-Peternell-Radloff-II}.

\begin{remark}
The varieties~\ref{dP-conic}, \xref{g=6:nonG}, and~\ref{conic-blowup}
are ordinary GM varieties in the sense of \cite{Debarre-Kuznetsov:GM}
(see Propositions~\ref{ex:dP-cb}, \ref{prop:descr:g=6:nonG}, 
and~\ref{prop:descr:conic-blowup},
respectively).
\end{remark}

Let $X$ be a Fano threefold with terminal Gorenstein singularities, $\uprho(X)=1$, and $\g(X)\ge 5$. 
Then $-K_X$ is very ample \cite[Theorem~4.2]{P:ratF-1}.
Therefore it is convenient to divide these varieties into 
the following classes:
\begin{enumerate}
\renewcommand\labelenumi{\rm (\Alph{enumi})}
\renewcommand\theenumi{\rm (\Alph{enumi})}
\item 
\label{cl-0}
the anticanonical image $X=X_{2g-2}\subset \PP^{g+1}$ is not an intersection of quadrics,
\item 
\label{cl-A}
$X$ does not belong to~\ref{cl-0} and 
$X$ contains a plane,
\item \label{cl-B}
$X$ does not belong to~\ref{cl-0} nor~\ref{cl-A}.
\end{enumerate}

The class~\ref{cl-0} contains exactly two families whose general members are nonrational and appear in 
\ref{g=5:trig} and~\ref{g=6:trig} (see \cite[Theorem~4.5]{P:ratF-1}).
The class~\ref{cl-A} contains exactly one nonrational family 
whose general members appear in 
\ref{g=6:nonG} (see \cite[Theorem~5.1]{P:ratF-1}).

If in the case~\ref{cl-B} we have $\rk\Cl(X)>1$, then
there is a $\QQ$-factorialization $\hat X\to X$ and a sequence of 
divisorial Mori contractions 
\[
\hat X =\hat X_0\longrightarrow\cdots \longrightarrow \hat X_n
\]
where each variety $\hat X_i$ is a weak Fano threefold 
whose anticanonical model $X_i$ is again of type~\ref{cl-B}
 and 
$\g(X_{i+1})\ge \g(X_{i})+1$ (see \cite[\S~3]{P:ratF-1} and \cite{P:degFc}).
Moreover, we may assume that $\uprho(\hat X_n)\le 3$, the variety $\hat X_n$ has no birational Mori contractions, and
has a Mori fiber space structure 
\[
\hat X_n\longrightarrow Z,
\]
where $Z$ is either a point, $\PP^1$, $\PP^2$ or 
$\PP^1\times \PP^1$. In these settings, 
the varieties $\hat X_n$ with $\uprho(\hat X_n)=2$ which are nonrational and are not Fano are described in 
Theorem~\ref{th:rho=2} (cases~\ref{dP-dP}, \ref{conic-conic}, \ref{dP-conic}).
The varieties $\hat X_n$ with $\uprho(\hat X_n)=3$ are described
by the following theorem.

\begin{theorem}
\label{th:rho=3}
Let $X$ be a nonrational Fano threefold with terminal Gorenstein singularities
with $\uprho(X)=1$, $\g(X)\ge 5$, and $\rk\Cl(X)>2$.
Assume that for any small 
$\QQ$-factorialization $X'\to X$
the variety $X'$ has no birational Mori contractions.
Then $\rk\Cl(X)=3$,
$\g(X)= 5$, and $X=X_8\subset\PP^6$ is a special intersection of three quadrics described in Proposition~\xref{ex:r=3}. 
\end{theorem}

The main theorems are proved in Sect.~\ref{sect:pfTh12} and~\ref{sect:pfTh14}.
In Sect.~\ref{sect:descr} we provide a detailed description of all varieties 
from the table in Theorem~\ref{th:rho=2}.

\subsection*{Acknowledgements.}
The author would like to thank Alexander Kuznetsov for fruitful discussions and 
the anonymous referee for careful reading the first version of the paper  and  useful comments.

\section{Preliminaries}

\subsection{Notation}
\label{r-not}
Throughout the paper we work over an algebraically closed field $\Bbbk$ 
of characteristic $0$. 

Let $\EEE =\oplus_{i=0}^n \OOO_{\PP^1}(a_i)$ be a vector bundle of rank $n+1$ 
on $\PP^1$ and let 
\[
\FF[a_1,\dots,a_n]:=\PP(\EEE)
\]
be its projectivization. 
We denote by $M$ and $F$ the tautological divisor and a fiber of the projection 
$\FF[a_1,\dots,a_n]\to \PP^1$, respectively.

Typically, we consider Fano threefolds with terminal Gorenstein singularities
but when we say this, it means that the singularities are not worse than that,
in particular, the variety can be smooth.
As usual, $\Pic(X)$ denotes the Picard group of a variety $X$ and 
$\uprho(X)$ is the rank of $\Pic(X)$.
$\Cl(X)$ denotes the Weil divisor class group of a normal variety~$X$.

Let $X$ be a Fano threefold with at worst Gorenstein terminal singularities. Then 
$\g(X):=\frac12(-K_X)^3+1$ is its \textit{genus}. 
The \textit{Fano index} $X$ is defined as follows 
\[
\iota(X):= \max \{ t\mid -K_X=tA,\quad A\in \Pic(X)\}.
\]
We say that $X$ is a del Pezzo threefold if $\iota(X)=2$. 
In this case, the integer
\[
\dd(X):=\frac18 (-K_X)^3
\]
is called the \textit{degree} of $X$.

\begin{proposition}[{\cite[Theorems~4.2 and~4.5]{P:ratF-1}}]
\label{prop:v-ample}
Let $X$ be a Fano threefold with terminal Gorenstein singularities and $g:=\g(X)\ge 5$. 
Assume that $X$ is not a del Pezzo threefold of degree $1$. Then the anticanonical class $-K_X$ is very ample and defines an embedding $X=X_{2g-2}\subset \PP^{g+1}$.
If furthermore $\rk\Cl(X)=2$ and $X$ is not of type~\ref{g=5:trig} nor~\ref{g=6:trig}, then $X$ is an intersection of quadrics.
\end{proposition}

\begin{scorollary}
\label{cor:v-ample}
In the notation of Theorem~\xref{th:rho=2}, assume that $X$ is not a del Pezzo threefold of degree $1$. Then the contractions $\tau$ and $\tau'$ coincide with anticanonical maps, that is, the divisor $-K_X$ is very ample, $-K_{\hat X}=\tau^*(-K_X)$ and $-K_{\hat X'}=\tau^{\prime *}(-K_X)$.
\end{scorollary}

We need some information about codimension two linear sections of the Grassmannian
$\Gr(2,5)\subset \PP^9$.

\begin{proposition}
\label{prop:dP5}
Let $G$ be the Grassmannian $\Gr(2,5)$ and let $G\subset \PP(\wedge^2\Bbbk^5)=\PP^9$ be its Pl\"ucker 
embedding. Let
$W:=\PP^7 \cap G\subset \PP^7$ be a codimension two linear section. 
Let $\mathscr{H}$ be the pencil of hyperplane sections passing through $W$.
Suppose that $\dim\Sing(W)\le 1$.
Then one of the following holds:
\begin{enumerate}
\item \label{prop:dP5:0}
$W$ is smooth and any member $H\in \mathscr{H}$
is smooth,
\item \label{prop:dP5:1}
$\Sing(W)$ is a line and there is exactly one singular member $H_1\in \mathscr{H}$ and this $H_1$ is a Schubert variety $\upsigma_{1.0}$,
\item \label{prop:dP5:2}
$\Sing(W)$ is a pair of skew lines and there are exactly two singular members $H_1,\, H_2\in \mathscr{H}$ and these $H_i$ are Schubert varieties $\upsigma_{1.0}$.
\end{enumerate}
\end{proposition}

\begin{proof}
Any hyperplane $\Lambda\subset \PP^9=\PP(\wedge^2\Bbbk^5)$ is given by 
a non-zero skew-symmetric bilinear form $q_\Lambda\in (\wedge^2\Bbbk^{5})^\vee$
and the hyperplane section $G\cap \Lambda$ is smooth if and only if $\rk (q_\Lambda)=4$.
If $\rk (q_\Lambda)=2$, then $G\cap \Lambda$ is a Schubert variety $\upsigma_{1,0}$.
Thus for any singular hyperplane section $H=G\cap \Lambda$ one can associate 
a two dimensional linear subspace $\PP(\ker(q_\Lambda))=\Sing(H)\subset \PP^4$.
From this, one can see that the Grassmannian $G=\Gr(2,5)$ is self-dual, i.e. its dual variety $G^\vee\subset \PP^{9\, \vee}$ is isomorphic to $\Gr(3,5)\simeq \Gr(2,5)$. Therefore, the pencil $\mathscr{P}$ of hyperplanes passing through $W$ either meets $G^\vee$ in at most $2$ points or is contained in $G^\vee$.

Suppose that $W$ is singular at $P_1\in W\subset G$.
Then some hyperplane $\Lambda_1\in \mathscr{P}$
contains the tangent space $T_{P_1,G}$. 
Then the hyperplane section $H_1:= \Lambda_1\cap G$ is a Schubert variety $\upsigma_{1,0}$, so it is singular
along a plane $\Pi_1$. 
Let $\Lambda_2\in \mathscr{P}$ be a hyperplane different from $\Lambda_1$.
Then $W=\Lambda_1\cap \Lambda_2\cap G$ is singular along the line $l_1:=\Pi_1\cap \Lambda_2$. 
Now assume that $W$ is singular at $P_2\in W\setminus l_1$. As above, we may assume that $\Lambda_2\supset T_{P_2,G}$ and so
the hyperplane section $H_2:= \Lambda_2\cap G$ is singular
along a plane $\Pi_2\ni P_2$. Then $W$ is singular along the line $l_2:=\Pi_2\cap \Lambda_1$. Since $\dim\Sing(W)\le 1$ by our assumption, the above arguments show 
that $\mathscr{P}\not\subset G^\vee$ and $\Sing(W)=l_1\cup l_2$. 
If $l_1\cap l_2\neq \varnothing$, then 
$\Pi_1\cap \Pi_2\neq \varnothing$ but in this case $\Pi_1\subset H_2$. Hence, $\Pi_1\subset W$ and $W$ is singular along 
$\Pi_1$, a contradiction.
\end{proof}

\begin{definition}
\label{def:DP5}
A four-dimensional Fano variety $W$ with Gorenstein terminal singularities 
such that $-K_W=3H$ and $H^4=5$, where $H$ is a Cartier divisor
is called \textit{quintic del Pezzo fourfold}. 
\end{definition}

The varieties described in Proposition~\ref{prop:dP5} are quintic del Pezzo fourfolds.

\begin{theorem}[{\cite{Fujita:DP-2}}]
\label{thm:DP5}
A smooth quintic del Pezzo fourfold is unique up to isomorphism and isomorphic to a smooth section of the Grassmannian 
$\Gr(2,5)\subset \PP^9$ by a linear subspace of codimension $2$, i.e. 
the variety from \xref{prop:dP5}\xref{prop:dP5:0}.
\end{theorem}

\section{Description of the varieties from Theorem~\ref{th:rho=2}}
\label{sect:descr}

\subsection{Case~\ref {dP-dP}}
\label{ex:dP-dP}
The threefold $\hat X$ has two structures of del Pezzo fibrations of degree $4$. Therefore, it can be embedded into a $\PP^4$-bundle over $\PP^1$. More precisely, $\hat X$ can be realized in $\FF[0^3, 1^2]$ as an intersection of two divisors $D_1,\, D_2\in |2M|$ \cite[2.11.7]{Takeuchi-2009}, 
\cite[7.1.7]{Jahnke-Peternell-Radloff-II}.

Below we provide another description of this variety.
\begin{sproposition}
\label{claim:case1-0}
The anticanonical image $X=X_8\subset \PP^6$ of a threefold of type~\ref 
{dP-dP} 
is a complete intersection of three quadrics so that one of them 
has corank $3$ \textup(i.e. it is singular along a plane\textup).
Conversely, a general complete intersection of three quadrics 
\begin{equation}
\label{eq:claim:case1-0}
X=Q_1\cap Q_2\cap 
Q_3\subset \PP^6, 
\end{equation} 
where $\corank(Q_1)=3$, is a threefold of type~\ref {dP-dP}.
\end{sproposition}
\begin{proof}
Since the variety $\hat X$ has a structure of del Pezzo fibration
of degree $4$, its canonical class is not divisible by $2$.
Hence, $\iota(X)=1$. According to Proposition~\ref{prop:v-ample}
the linear system $-K_X$ is very ample.
Let $\hat{F}\subset \hat X$ be a general fiber of $f$ and let $F:=\tau(\hat{F})$. 
Then $\hat{F}$ is a quartic del Pezzo surface 
and it does not contain flopping curves.
Hence the map $\tau_{F}: \hat{F}\to F\subset X\subset \PP^6$ is 
finite and birational. Clearly, $\tau_{F}$ is 
given by the linear system $|-K_{\hat X}|\bigr|_{\hat{F}}\subset |-K_{\hat{F}}|$ (see Corollary~\ref{cor:v-ample}).
Hence, $F$ is a surface of degree $4$. By
\cite[Theorem~4.5]{P:ratF-1} the variety $X=X_8\subset \PP^6$ is a complete intersection of three quadrics, see \eqref{eq:claim:case1-0}.
Thus, we have $\dim \langle F\rangle=4$. So, $|-K_{\hat X}|\bigr|_{\hat{F}}= |-K_{\hat{F}}|$ is a complete 
linear system and $\tau_{F}: \hat{F}\to F$
is an isomorphism.
Therefore, $F$ is a complete intersection of two quadrics in 
$\langle F\rangle=\PP^4$ and so one of the quadrics in the net generated by $Q_1, Q_2, 
Q_3$ contains $\langle F\rangle=\PP^4$.
Let $Q_1\supset\langle F\rangle$. Then $\dim \Sing(Q_1)\ge 2$.
Since $X$ has only isolated singularities, $\Sing(Q_1)$ is a 
plane.

Conversely, let $X=Q_1\cap Q_2\cap Q_3\subset \PP^6$, 
where $\Sing(Q_1)$ is a plane $\Pi$. 
For general choice of $Q_2$ and $Q_3$, the singular locus of $X$ consists of 
exactly four nodes at $\Pi\cap Q_2\cap Q_3$.
The quadric $Q_1$ contains two pencils of four-dimensional 
linear subspaces 
$\Lambda_t$ and $\Lambda_t'$, $t\in \PP^1$.
Put 
\[
F_t:=\Lambda_t\cap Q_2\cap Q_3,\qquad F_t':=\Lambda_t'\cap Q_2\cap Q_3.
\]
Then $\{F_t\}$ and $\{F_t'\}$ are two pencils of quartic del 
Pezzo surfaces which are not Cartier divisors on $X$.
By blowing up the smooth surface $F_t'$ we obtain a 
small resolution 
$\tau: \hat X\to X$.
The base locus of the linear system $|F_t|$ on $X$ coincide with $\Sing(X)=\Pi\cap Q_2\cap Q_3$. Hence the corresponding proper transform $|\hat F_t|$ on $\hat X$ is base point free and defines a structure of quartic del Pezzo fibration $\hat X\to \PP^1$ as in \eqref{eq:diagram}. Similarly, by blowing up $F_t$ we obtain
the right hand side of the diagram \eqref{eq:diagram}.
\end{proof}

\subsection{Case~\ref{blowup-V14}}
In this case the varieties $Y$ and $Y'$ must be smooth by Theorem~\ref{th:rho=1}. The link \eqref{eq:diagram} is described in \cite{Takeuchi-1989}. As in the case~\ref{dP-dP} we have

\begin{sproposition}
The anticanonical image $X=X_8\subset \PP^6$ of a threefold of type 
\xref{blowup-V14}
is a complete intersection of three quadrics 
containing a union $S\cup S'$ of two smooth rational
quartic scrolls such that the linear span of $S\cup S'$ 
is five-dimensional and 
$S\cap S'$ is a rational curve of degree $6$.

Conversely, a general complete intersection of three quadrics in $\PP^6$ containing a smooth 
quartic scroll $S=S_4\subset \PP^5$ is a threefold of type \xref{blowup-V14}.
\end{sproposition}

\begin{proof}
Similarly to Proposition~\ref{claim:case1-0}
the threefold $X=X_8\subset \PP^6$ is a complete intersection of three quadrics.
Recall \cite{KPS:Hilb} that if $C$ is a non-degenerate conic on a Fano threefold $Y=Y_{14}\subset \PP^9$, then 
its normal bundle has the form 
\[
\NNN_{C/Y}\simeq \OOO_{\PP^1}(a)\oplus \OOO_{\PP^1}(-a),\qquad a=0\quad \text{or}\quad 1.
\]
Hence, for the exceptional divisor $\hat E$ of the blowup $f:\hat X\to Y$ 
one has $\hat E\simeq \PP^1\times \PP^1$ or $\FF_2$. 
Let $S:=\tau(\hat E)$. This surface is ruled by lines.
Since $(-K_{\hat X})^2\cdot \hat E=4$ and $\tau$ has connected fibers,
$\deg S=4$. Since $\tau$ contracts only a finite number of curves, 
the singularities of $S$ are isolated. If $\dim \langle S\rangle\le 4$, then $S$ is an intersection of two quadrics in $\PP^4$.
In this case $S$ must be normal and so it cannot be ruled by lines.
Hence, $\dim \langle S\rangle=5$ and 
$S$ is a rational quartic scroll in $\PP^5$. 
Similarly, for the exceptional divisor $\hat E'$ of $f'$ we have 
$(-K_{\hat X'})^2\cdot \hat E'=4$. Hence, $S':=\tau'(\hat E')\subset X\subset \PP^5$ is also a rational quartic scroll. 
It is easy to see that $S+S'\sim -K_X$ and 
$-K_X\cdot (S\cap S')= -K_{\hat X}\cdot \hat E\cdot \chi^{-1}_* \hat E'=6$.
Hence, $\deg (S\cap S')=6$.

Conversely, let $S=S_4\subset \PP^5$ be the image of $\PP^1\times \PP^1$ under the map 
defined by the linear system 
of bidegree $(1,2)$. The homogeneous ideal $I$ of $S$ is generated by 6 
quadrics. 
Take three general quadrics $q_i(x_1,\dots,x_6)\in I$, $i=1,2,3$ and let 
$X\subset \PP^6$ be 
the variety defined by three equations
\[
x_0^2+ q_i(x_1,\dots,x_6)=0,\qquad i=1,2,3.
\]
The intersection $X\cap \langle S\rangle=S\cup S'$ where $S'$ is a residual 
(smooth) surface of degree $4$. The variety $X$ has exactly 10 nodes.
As in the proof of Proposition ~\ref {claim:case1-0} one can show that by blowing up $S$ and $S'$ we obtain two small resolutions $\tau :\hat X\to X$ and 
$\tau' :\hat X'\to X$ satisfying the desired properties.
\end{proof}

\subsection{Case~\ref{blowup-W2-point}}
In this case $X$ is a del Pezzo threefold of degree $1$.
\begin{sproposition}
\label{prop:blowup-W2-point}
Any Fano threefold $X$ of type~\ref{blowup-W2-point} can be given in the weighted projective space $\PP(1,1,1,2,3)$ by the equation
\begin{equation}
\label{eq-blowup-W2-point}
z^2+y^3+ y^2 \phi(x_1,x_2,x_3)+ y s(x_1,x_2,x_3) + q(x_1,x_2,x_3)^2=0
\end{equation} 
where $x_i$, $y$, and $z$ are quasi-homogeneous coordinates on $\PP(1,1,1,2,3)$ with $\deg x_i=1$, $\deg y=2$, $\deg z=3$, and $\phi$, $s$ and $q$ are quasi-homogeneous forms of degree $2$, $4$, and $3$, respectively.

Conversely, a variety given by a general equation of the form \eqref{eq-blowup-W2-point} is a threefold of type \xref{blowup-W2-point}.
\end{sproposition}

\begin{proof}
The divisor $-K_{\hat X}$ is divisible by $2$. Thus $\iota(X)=2$ and $X$ is a del Pezzo threefold of degree $1$. It is well-known that in this case 
$X$ can be given in $\PP(1,1,1,2,3)$ by a quasi-homogeneous polynomial of degree $6$ (see e.g. \cite{Shin1989}). Let $x_i$, $y$, and $z$ are quasi-homogeneous coordinates on $\PP(1,1,1,2,3)$ with $\deg x_i=1$, $\deg y=2$, $\deg z=3$. 
Since the singularities of $X$ are terminal Gorenstein, $X$ does not contain 
the points $(0{:}0{:}0{:}1{:}0)$ and $(0{:}0{:}0{:}0{:}1)$.
Hence the equation of $X$ must contain the terms $z^2$ and $y^3$ and so it can be written in the form 
\[
z^2+y^3+ y^2 \phi(x_1,x_2,x_3)+ y s(x_1,x_2,x_3) + \psi(x_1,x_2,x_3)=0,
\]
where $\deg \phi=2$, $\deg s=4$ and $\deg \psi=6$.

Let $E\subset \hat X$ (resp., $E'\subset \hat X'$) be the exceptional divisor of $f$ (resp. $f'$), and let $D:=\tau(E)$, $D':=\tau'(E')$.
It is easy to compute that $D+D'\sim -K_X$ (see e.g. \cite[Theorem~5.3]{P:GFano1}). Hence $D+D'$ is cut out on $X$ by a quadratic equation 
$\lambda y+\theta(x_1,x_2,x_3)=0$, where $\lambda\neq 0$ because the intersection is irreducible. After a suitable coordinate change we may assume that $D+D'=X\cap \{y=0\}$. Then $\psi$ must be a square. 
The last assertion can be proved similarly to Proposition~\ref{claim:case1-0}.
\end{proof}

\begin{sremark}
\label{rem:V1}
\begin{enumerate}
\item\label{rem:V1a} 
Note that the threefold $X$ admits a natural involution $z\mapsto -z$ which is called the \textit{Bertini involution}. This involution switches the left hand and the right hand sides of the diagram \eqref{eq:diagram}. In particular, $Y\simeq Y'$.
\item \label{rem:V1b}
Since $\hat X$ is $\QQ$-factorial and $\uprho(\hat X)=2$, the threefold $Y$ must be $\QQ$-factorial as well. 
The same holds for $Y'$.
\end{enumerate}
\end{sremark}

\subsection{Case~\ref{conic-conic}} 
\label{ssect:conic-conic}
A general variety as in~\ref{conic-conic} can be constructed as follows.
Let $V=V_3\subset \PP^4$ be a smooth cubic hypersurface and let $S\subset V$ be 
its smooth hyperplane section. Take a pair of skew lines $l_1,\, l_2$ on $S$.
There exists a four-dimensional Sarkisov link
\begin{equation*}
\vcenter{
\xymatrix@R=1em{
&\tilde \PP^4\ar[dr]^{}\ar[dl]^{}&
\\
\PP^4&&\PP^2\times \PP^2
}}
\end{equation*}
where $\tilde \PP^4\to \PP^4$ is the blowup of $l_1\cup l_2$ and 
$\tilde \PP^4\to \PP^2\times \PP^2$ is the contraction of the proper transform of 
the linear span $\langle l_1\cup l_2\rangle$ to $\PP^1\times \PP^1\subset \PP^2\times \PP^2$.
Taking the proper transform 
$\tilde V\subset \tilde \PP^4$ of our cubic 
$V=V_3\subset \PP^4$ we obtain the left hand side of the following three-dimensional link
(see e.g. \cite[Proposition~13.4]{P:G-MMP})
\begin{equation}
\label{eq:diagram3a}
\vcenter{
\xymatrix@R=1em{
&\tilde{V}\ar[dr]^{} \ar[dl]_{\sigma}\ar@{-->}[rr]^{\chi}&&\tilde{X}\ar[dr]^{\tilde{\tau}}
\ar[dl]_{} 
\\
V&&W&&X
}}
\end{equation}
Here $\sigma: \tilde V\to V$ is the blowup of 
$l_1\cup l_2$, $W$ is a divisor of bidegree $(2,2)$ on $\PP^2\times \PP^2$
having $5$ nodes, $\tilde V\to W$ is the anticanonical map, $\chi : \tilde V \dashrightarrow \tilde X$ is a flop, and $\tilde X\to X$ is 
a contraction of a divisor $E\simeq \PP^1\times \PP^1$ to an ordinary double 
point $P\in X$. 
The variety $X$ is a (singular) Fano threefold 
of genus $8$.
The point $P\in X$ has two small resolutions $\tau: \hat X\to X$ 
and $\tau': \hat X'\to X$ so that $\tilde X\to X$ passes through both of them.
Thus the right hand side of the diagram \eqref{eq:diagram3a} can be completed to the following one:
\begin{equation}
\label{eq:diagram3}
\vcenter{
\xymatrix@R=2em@C=3em{
&\tilde X\ar[dr]\ar[dl]\ar[d]\ar@/^1.2em/[dd]^(.4){\tilde{\tau}}&
\\
\hat X\ar[d]_{f}\ar[dr]^(.3){\tau}|!{[d];[r]}\hole &W\ar[dl]^{}\ar[dr]^{} & \hat X'\ar[d]^{f'}\ar[dl]_(.3){\tau'}|!{[d];[l]}\hole
\\
\PP^2&X&\PP^2
}}
\end{equation}
where the morphisms $\PP^2\times \PP^2\supset W\to \PP^2$ are projections to the factors of $\PP^2\times \PP^2$. Removing $\tilde X$ and $W$ in \eqref{eq:diagram3} we obtain the diagram \eqref{eq:diagram} of type \eqref{ssect:conic-conic}. 

Note that the varieties $X=X_{14}\subset \PP^9$ of this type and the corresponding link are missing in \cite[7.5]{Jahnke-Peternell-Radloff-II}.
The reason is that the statement of \cite[Theorem~7.4]{Jahnke-Peternell-Radloff-II} is not correct.

\subsection{Case~\ref{blowup-cubic-point}}
In this case $\iota(X)=2$ and $X$ is a del Pezzo threefold of degree $2$. 

\begin{sproposition}
Any Fano threefold $X$ of type~\ref{blowup-cubic-point} can be given in the weighted projective space $\PP(1,1,1,1,2)$ by the equation
\begin{equation}
\label{eq-blowup-cubic-point}
y^2=q(x_1,\dots,x_4)^2-4\ell(x_1,\dots,x_4)s(x_1,\dots,x_4),
\end{equation} 
where $x_i$ and $y$ are quasi-homogeneous coordinates on $\PP(1,1,1,1,2)$ with $\deg x_i=1$, $\deg y=2$ and $\ell$, $q$ and $s$ are quasi-homogeneous forms of degree $1$, $2$, and $3$, respectively.

Conversely, the variety given by a general equation of the form \eqref{eq-blowup-cubic-point} is a threefold of type \xref{blowup-cubic-point}.
\end{sproposition}

\begin{proof}
Similarly to Proposition~\ref{prop:blowup-W2-point} we see that $X$ is a del Pezzo threefold of degree $2$ and it can be given in $\PP(1,1,1,1,2)$ by an equation
\[
y^2=\psi(x_1,x_2,x_3,x_4),
\]
where $\deg \psi=4$. Again, let $E\subset \hat X$ and $E'\subset \hat X'$) be exceptional divisors, and let $D:=\tau(E)$, $D':=\tau'(E')$.
Then $D+D'\sim -\frac 12 K_X$ (see e.g. \cite[Theorem~5.3]{P:GFano1}). Hence $D+D'$ is cut out on $X$ by a linear equation $\ell(x_1,\dots,x_4)=0$. 
Since $X\cap \{\ell=0\}$ is reducible, we see that $\psi$ has the form as in 
\eqref{eq-blowup-cubic-point}.
\end{proof}

\begin{sremark}
\label{rem:V2} 
\begin{enumerate}
\item \label{rem:V2a}
The threefold $X$ admits a natural involution $y\mapsto -y$ which is called the \textit{Geiser involution}. This involution switches the left hand and the right hand sides of the diagram \eqref{eq:diagram}.
In particular, $Y\simeq Y'$.

\item \label{rem:V2b}
In suitable coordinates for $Y\subset \PP^4$ and $Y'\subset \PP^4$ the maps $X\dashrightarrow Y$ and $X\dashrightarrow Y'$ can be written as follows
\[
\bigl(y,\, x_1,\dots,x_4\bigr) \longmapsto \bigl(-q\pm y,\, 2\ell x_1,\dots,2lx_4\bigr).
\]
Then the equation of the cubic 
$Y\subset \PP^4$ (and $Y'\subset \PP^4$)
has the form
\[
x_0^2\ell (x_1,\dots,x_4)+x_0 q(x_1,\dots,x_4)+s(x_1,\dots,x_4)=0.
\]
\end{enumerate}
\end{sremark}

\subsection{Case~\ref{g=5:trig}}
The threefold $\hat X$ has a structure of a del Pezzo fibration of degree $3$. Therefore, it can be embedded into a $\PP^3$-bundle over $\PP^1$. More precisely, $\hat X$ can be realized in $\FF[0, 1^3]$ as a divisor $D\in |3M-F|$.
For details we refer to \cite[2.9.4]{Takeuchi-2009}, \cite[Proposition~5.2]{Jahnke-Peternell-Radloff-II} (see also \cite[Example~4.6]{P:ratF-1}
\cite[Theorem~1.6, case~$T_3$]{Przhiyalkovskij-Cheltsov-Shramov-2005en},
\cite[Table~2, Case~3]{BrownCortiZucconi-2004}).
This provides the following description of the anticanonical model of $X$.

\subsection*{Construction}
Consider the cone $\Cone(V)\subset\PP^6$ over the Segre variety $V=V_3\subset\PP^5$, 
$V\simeq\PP^2\times\PP^1$. Let $O\in \Cone(V)$ be its vertex.
The variety $\FF[0, 1^3]$ can be regarded as a small resolution of $\Cone(V)$.
Let $\Pi\subset \Cone(V)$ be a three-dimensional linear subspace. 
Consider the cubic hypersurface $Z=Z_3\subset\PP^6$ passing through $\Pi$ and 
let $X=X_8\subset\PP^6$ be a residual subvariety of $\Cone(V)\cap Z$ to $\Pi$, that is, 
$\Cone(V)\cap Z=X\cup \Pi$. Then $X$ is an anticanonically embedded Fano threefold of 
genus 5. 
Moreover, the quadrics in $\PP^6$ passing through $X$ cut out $\Cone(V)\subset\PP^6$. 
For a general choice of $Z$, the singular locus of $X$ consists of a single 
ordinary double point at $O$ and $\uprho(X)=1$.

\subsection{Case~\ref{g=5:conic-G}}
The varieties of this type appeared in 
\cite[Proposition~7.11]{Jahnke-Peternell-Radloff-II} and \cite[Theorem~2, case~$8^o$]{P:factorial-Fano:e}. 
Let $x_0,\dots, x_6$ be homogeneous coordinates of $\PP^6$, $l_0,\dots,l_5$ 
general linear forms and $q_3$ be a general quadratic form. Then the complete intersection 
of 
\[
q_1 = x_0l_0 + x_1l_1 + x_2l_2,\quad q_2 = x_0l_3 + x_1l_4 + x_2l_5, \quad q_3
\]
is a Fano threefold $X$ containing the quadric surface 
\[
\{x_0 = x_1 = x_2 =q_3= 0\}. 
\]

\subsection{Case~\ref{dP-conic}} 
According to \cite[2.11.6]{Takeuchi-2009} 
the threefold $\hat X$ can be realized in 
$\FF[0^2, 1^3]$ as intersection 
of two divisors $D_1\in |2M-F|$ and $D_2\in |2M|$.

\begin{sproposition}[{\cite[2.11.6]{Takeuchi-2009}}]
\label{ex:dP-cb}
In the case~\ref{dP-conic} the threefold $X$ is an intersection of 
a singular quintic del Pezzo fourfold $W=W_5\subset \PP^7$ of type \xref{prop:dP5}\xref{prop:dP5:1} or~\xref{prop:dP5:2} 
with a quadric.
\end{sproposition}

\subsection{Case~\ref{g=6:nonG}}
In this case the left hand side of the diagram \eqref{eq:diagram} is described in \cite[Example~5.2]{P:ratF-1}.
By Corollary~\ref{cb:P2} the conic bundle $f:$
The variety $Y'=Y_{21/2}$ is a non-Gorenstein Fano threefold with 
$(-K_{Y'})^3=21/2$
whose singular locus consists of one point of type $\frac12(1,1,1)$.
It appears in the Graded Ring Database \cite{GRD} as the entry no.
29375.

Recall that the smooth quintic del Pezzo fourfold $W=W_5\subset \PP^7$ contains two types of planes (see \cite{Fujita:DP-2} or \cite{PZ:g10}): 
\begin{itemize}
\item
a one-dimensional family $\Pi_t$ with $\Pi_t^2=1$;
\item
a single plane $\Xi$ with $\Xi^2=2$. 
\end{itemize}

\begin{sproposition}
\label{prop:descr:g=6:nonG}
In the case \xref{g=6:nonG} the threefold $X$ is an intersection of 
a smooth quintic del Pezzo fourfold $W=W_5\subset \PP^7$ with a quadric $Q$ containing the plane $\Xi\subset W$. 
\end{sproposition}

\begin{proof}
Recall that there exists the following Sarkisov link (see \cite{Fujita:DP-2}, \cite{PZ:g10}):
\[
\xymatrix{
&\tilde \PP^4\ar[dl]_{\varsigma}\ar[dr]^{\varrho}&
\\
\PP^4&& W
} 
\]
where $\varsigma$ is the blowup of a rational twisted cubic curve $C=C_3\subset \PP^4$
and $\varrho$ is the blowup of the plane $\Xi$. 
Now, let $H_{\PP^4}$ be a hyperplane in $\PP^4$ and
let $E_{\varsigma}$ be the $\varsigma$-exceptional divisor.
Then the morphism $\varrho: \tilde \PP^4\to W=W_5\subset \PP^7$
is given by the linear system $|2\varsigma^*H_{\PP^4}-E_{\varsigma}|$
and the $\varrho$-exceptional divisor $E_{\varrho}$ is linearly equivalent to 
$\varsigma^*H_{\PP^4}-E_{\varsigma}$.
Let $Y=Y_3\subset \PP^4$ be a smooth cubic containing $C$
and let $\tilde Y\subset \tilde \PP^4$ be its proper transform.
Then 
\[
\tilde Y\sim 3\varsigma^*H_{\PP^4}-E_{\varsigma}\sim 2(2\varsigma^*H_{\PP^4}-E_{\varsigma}) -E_{\varrho}.
\]
Hence $\varrho(\tilde Y)$ is cut out on $W$ by a quadric containing 
$\varrho(E_{\varrho})=\Xi$.
\end{proof}

\subsection{Case~\ref{conic-blowup}} 
In this case the variety $Y=Y_{14}\subset \PP^9$ is a smooth Fano threefold with $\iota(Y)=1$ and $\g(Y)=8$. 

\begin{sproposition}
\label{prop:descr:conic-blowup}
In the case \xref{conic-blowup} the threefold $X$ is an intersection of 
a \textup(possibly singular\textup) quintic del Pezzo fourfold $W=W_5\subset \PP^7$,
which is a codimension two linear section of the the Grassmannian $\Gr(2,5)\subset \PP(\wedge^2\Bbbk^5)=\PP^9$
\textup(see Proposition~\xref{prop:dP5}\textup),
with a quadric $Q$ containing a cubic scroll $S=S_3\subset W$.
\end{sproposition}
Here by cubic scroll we mean a surface $S=S_3\subset \PP^4$ of degree $3$ that is not contained in a hyperplane. 
Such a surface $S$ is either smooth (and isomorphic to the Hirzebruch surface $\FF_1$) or a cone over a rational twisted cubic curve.

\begin{proof} 
The morphism $f:\hat X\to Y$ is the blowup of a line $l\subset Y$ and the map $Y \dashrightarrow X=X_{10}\subset \PP^7$ is the projection from $l$ \cite[Corollary~4.3.2]{IP99}.
It is known that the variety $Y=Y_{14}\subset \PP^9$ can be realized as a 
codimension $5$ linear section of the Grassmannian $\Gr(2,6)\subset \PP^{14}$ \cite{Gushel:g8}, \cite{Debarre-Kuznetsov:GM}.
We regard $\Gr(2,6)$ as the family of lines in $\PP^5$.
Under this description, the line $l$ is a Schubert variety $\upsigma_{4,3}$.
There is exactly one 
four-dimensional linear subspace $\Pi\subset \Gr(2,6)$ (a Schubert variety $\upsigma_{4,0}$) containing $l$. The points on $\Pi$ are represented by 
lines passing through a fixed point $P\in \PP^5$.
The projection $\PP^5 \dashrightarrow \PP^4$ from $P$ induces 
the projection 
\[
\Psi: \Gr(2,6)\subset \PP^{14} \dashrightarrow \Gr(2,5)\subset \PP^9
\]
from $\Pi$ whose fibers are planes. By the construction, the image $\Psi(Y)$ is contained in a 
codimension $2$ linear section $W=W_5=\Gr(2,5)\cap \PP^7\subset \PP^{9}$. Since $Y\cap \Pi=l$, we have the following diagram whose restriction to $Y$ gives the left hand side of \eqref{eq:diagram}
\[
\xymatrix{
&
\hat X\ar[dl]_f\ar@{} !<-0pt,0pt>;[r]!<-0pt,0pt> |-*[@]{\subset}\ar@/^3em/[drrr] ^{\tau}
&
\widetilde \Gr(2,6)\ar[dr]^{\tau_G}\ar[dl]_{f_G}&
\\
Y\ar@{} !<-0pt,0pt>;[r]!<-0pt,0pt> |-*[@]{\subset}
&
\Gr(2,6)\ar@{-->}[rr]^{\Psi}
& &
\Gr(2,5) &X\ar@{} !<-0pt,0pt>;[l]!<-0pt,0pt> |-*[@]{\supset}
} 
\]
Here $\tau_G$ is a $\PP^2$-bundle and $X=\Psi(Y)\subset W$.
Recall that $X$ is an intersection of quadrics (see Proposition~\ref{prop:v-ample}). Hence there is a quadric $Q\subset \PP^7$ such that $X\subset W\cap Q$. Since $\deg X=10$, we have $X= W\cap Q$.
\end{proof}

\subsection{Case~\ref{g=6:trig}}
The Fano threefold of type~\ref{g=6:trig} is trigonal. 
It was considered in a number of papers, see 
\cite[\S~4.4, table~2]{BrownCortiZucconi-2004},
\cite[1.11]{Przhiyalkovskij-Cheltsov-Shramov-2005en},
\cite[2.9.3]{Takeuchi-2009},
\cite[Example~4.7]{P:ratF-1}. By \cite[2.9.3]{Takeuchi-2009}
the threefold $\hat X'$ in \eqref{eq:diagram} can be realized in
$\FF[0, 1^2, 2]$ as a member of the linear system 
$D_1\in |3M-2F|$.

For convenience of the reader we provide the detailed description.
\begin{sproposition}
Let $X$ be a variety of type~\ref{g=6:trig}. Then
the anticanonical image of $X$ is a divisor of the cone 
\[
\Cone(\FF[1^2,2]) \subset \PP^7
\]
over the $\PP^2$-bundle $\FF[1^2,2]$, and is the irreducible component of a cubic hypersurface $X +\Pi_1 +\Pi_2$ through the vertex, where $\Pi_i$ is the cone $\Cone(f) \subset \Cone(\FF[1^2,2])$ over a fiber $f$ of $\FF[1^2,2] \to \PP^1$; the singularity of $X$ is one node, the vertex of the cone.
\end{sproposition}

The del Pezzo threefold $Y'$ in the diagram \eqref{eq:diagram} must be locally factorial (or smooth) and center of the 
blowup must be contained in the smooth locus.

\subsection{Case~\ref{dP-blowup}}
By \cite[2.11.4]{Takeuchi-2009}
the threefold $\hat X'$ of type~\ref {dP-blowup} can be realized in
$\FF[0, 1^3, 2]$ as intersection 
$D_1\cap D_2$, $D_1\in |2M-2F|$, $D_2\in |2M-F|$.

\begin{sproposition}
The anticanonical model $X$ is embedded to $\Cone(\FF[1^3,2]) \subset \PP^9$, and is the intersection $W_1 \cap W_2$ of the divisors $W_i \subset \Cone(\FF[1^3,2])$ through the vertex, where $ W_1 + \Pi_1 + \Pi_2$, $W_2 + \Pi_3$ are quadric hypersurface sections of $\Cone(\FF[1^3,2])$ for the cone $\Pi_i = \Cone(\PP^3) \subset \Cone(\FF[1^3,2])$ over a fiber of $\FF[1^3,2] \to \PP^1$; the singularity of $V$ is one node corresponding to the vertex of $\Cone(\FF[1^3,2])$.
\end{sproposition}

\section{Nonrationality criteria}

\subsection{Conic bundles}
Let $V$ be a threefold with terminal Gorenstein singularities.
A \textit{conic bundle} is a flat contraction $f:V\to S$ to a surface such that the anticanonical divisor
$-K_V$ is $f$-ample. In this situation any scheme fiber $V_s$, $s\in S$ is a plane conic.
A conic bundle $f:V\to S$ is said to be \textit{standard} if both $V$ and $S$ are smooth and the relative Picard group $\Pic(V/S)$ has rank $1$. 
The points $s\in S$ such that the corresponding fiber $V_s$ is singular 
form a reduced divisor. It is called the \textit{discriminant curve}. 

\begin{sproposition}[{\cite[Theorem~1]{Iskovskikh:Duke}}, see also {\cite[Proposition~5.6]{P:rat-cb:e}}]
\label{prop:conic-bundle:rat}
Let $f:V\to S$ be a conic bundle over a projective rational surface $S$ 
and let $\Delta\subset S$ be the discriminant curve. 
Suppose that there exists a contraction $\phi: S\to T$ to a curve 
such that for a fiber $S_t$ over a closed point $t\in T$ one has
$\Delta\cdot S_t\le 3$. Then $V$ is rational.
\end{sproposition}

The following result was proved in 
\cite[Sect.~4]{Beauville:Prym} and \cite[Sect.~10]{Shokurov:Prym}.
\begin{stheorem}
Let $f:V\to S$ be a standard conic bundle over a projective rational surface $S$. 
Assume that $S$ is minimal, i.~e. it does not contain $(-1)$-curves.
The variety $V$ is rational if and only if its intermediate Jacobian $\J(V)$ is 
isomorphic as a principally polarized Abelian variety 
to a sum of the Jacobians of nonsingular curves, i.e., if its Griffiths component 
$\JG (V)$ is trivial.
\end{stheorem}

\begin{scorollary}[{\cite[Sect.~4]{Beauville:Prym}}, {\cite[\S~10]{Shokurov:Prym}}]
\label{cb:P2-0}
Let $f:V\to\PP^2$ be a standard conic bundle and let $\Delta\subset \PP^2$ be
the discriminant curve.
If $\deg\Delta\ge 6$, then $V$ is not rational.
\end{scorollary}

\begin{scorollary}
\label{cb:P2}
Let $V$ be a variety with only terminal $\QQ$-factorial singularities and let
$f:V\to \PP^2$ be a Mori contraction.
Let $\Delta\subset \PP^2$ be
the discriminant curve. 
If the variety $V$ is not rational, then $\deg \Delta\ge 5$.
Moreover, if $\deg \Delta=5$, then $f$ is a standard conic bundle and
the double cover $\tilde \Delta\to \Delta$ is given by an odd theta-characteristic.
\end{scorollary}

\begin{proof}
Suppose that $f$ is not a standard conic bundle.
Thus $V$ is singular at some point of a fiber $f^{-1}(o)$, $o\in \PP^2$.
Let $f^\bullet : V^\bullet\to S^\bullet$ be a standard model of 
$f:V\to \PP^2$ \cite{Sarkisov:82e}
and let $\Delta^\bullet$ be the corresponding discriminant curve.
We may assume that $S^\bullet$ dominates $\PP^2$.
Thus we have a birational morphism $\alpha: S^\bullet\to \PP^2$ so that
\begin{equation*}
\Delta^\bullet\le (\alpha^*\Delta)_{\red}
\quad \text{and}\quad
\alpha(\Delta^\bullet)=\Delta
\end{equation*}
(see e.g. \cite[Corollary~3.3.6]{P:rat-cb:e}).
Let $\LLL$ be the pencil of lines on $\PP^2$ passing through $o$
and let $\LLL^\bullet$ be its proper transform on $S^\bullet$. Assume that $\deg \Delta\le 5$.
Then by the projection formula 
\[
\LLL^\bullet \cdot \alpha^* \Delta=
\LLL \cdot \Delta\le 5.
\]
By \cite[Corollary~10.8]{P:rat-cb:e} the point $o$ is singular on $\Delta$.
Let $m_o$ the multiplicity of $\Delta$ at $o$. Then a general member of $\LLL^\bullet$
meets a component of $\alpha^* \Delta$ whose coefficient equals $m_o$.
Hence, 
\[
\LLL^\bullet \cdot \Delta^\bullet \le \LLL^\bullet \cdot \alpha^* \Delta-(m_o-1)\le 6-m_o.
\]
If $m_o\ge 3$, then $\LLL^\bullet\cdot\Delta^\bullet\le 3$ and 
$V^\bullet$ is rational by 
Proposition~\ref{prop:conic-bundle:rat}. 
Thus we may assume that 
$m_o=2$. Then by \cite[Lemma~10.10]{P:rat-cb:e} the curves
$\alpha^{-1}(o)$ and $\Delta^\bullet$ have no common components.
Again 
\[
\LLL^\bullet 
\cdot\Delta^\bullet\le \LLL^\bullet \cdot \alpha^* \Delta -m_o\le3
\]
and $V^\bullet$ is rational again by 
Proposition~\ref{prop:conic-bundle:rat}.
Thus we may assume that $f$ is a standard conic bundle.
If the corresponding double cover of $\Delta$ is defined by an even theta-characteristic,
then $V$ is rational by \cite{Panin1980} (see also \cite[Theorem~1]{Iskovskikh:Duke}, \cite[Proposition~8.1]{P:rat-cb:e}). 
\end{proof}

\subsection{Del Pezzo fibrations}
Let $V$ be a threefold with terminal Gorenstein singularities.
A \textit{del Pezzo fibration} is a contraction $f:V\to B$ to a curve such 
that the anticanonical divisor
$-K_V$ is $f$-ample. In this situation the generic fiber $V_\eta$ is a smooth del Pezzo surface. The \textit{degree} of a del Pezzo fibration $f:V\to B$ is the 
anticanonical degree of $V_\eta$.

\begin{slemma}
\label{lemma:del-pezzo}
Let $f:V\to B$ be a del Pezzo fibration of degree $d$. Assume that $V$ is rationally connected but not rational.
Then $B\simeq\PP^1$ and $d\le 4$.
\end{slemma}
\begin{proof}
Let $V_\eta$ be the scheme generic fiber. Since $\Bbbk(B)$ is a 
\type{c_1}-field, $V_\eta$ has a $\Bbbk(B)$-point.
If $K_{V_\eta}^2\ge 5$, then the surface $V_\eta$ is rational over 
$\Bbbk(B)$ according to \cite[Ch.~IV, \S~7]{Manin:book:74}.
\end{proof}

Denote by $\Eu(V)$ the topological Euler characteristic of $V$.

\begin{stheorem}[{\cite{Alekseev:87}}]
\label{thm:DPal}
Let $\pi: V\to \PP^1$ be a standard
del Pezzo fibration of degree $4$
such that $V$ is smooth and the relative Picard group $\Pic(V/B)$ has rank $1$.
\begin{enumerate}
\item 
If $\Eu(V)\notin \{0,\, -8,\, -4\}$ then $V$ is not rational.
\item 
If $\Eu(V)\in \{0,\, -8\}$ then $V$ is rational.
\item 
If $\Eu(V)=-4$ then $V$ is rational if and only if the Griffiths component of 
its intermediate Jacobian is trivial.
\end{enumerate}
\end{stheorem}

Note that a smooth complete intersection $V\subset \FF(a_0,\dots,a_4)$ of 
divisors $D_1\sim 2 M+b_1F$ and $D_2\sim 2 M+b_2F$ is a del Pezzo fibration as in Theorem~\ref{thm:DPal}.
The following fact is taken from \cite{Shramov2006}.
For convenience of the reader we provide a complete proof.

\begin{slemma}
\label{lemma:Eu}
Let $V\subset \FF(a_0,\dots,a_4)$ be a smooth complete intersection of 
divisors $D_1\sim 2 M+b_1F$ and $D_2\sim 2 M+b_2F$, where
$M$ is the tautological divisor and $F$ is the class of the fiber. Then we have
\[
\Eu(V) = 16-16 \sum a_i - 20\sum b_i.
\]
\end{slemma}
\begin{proof}
Let $\mathscr{T}_{V}$ be the tangent bundle on $V$ and let $\NNN_{V/\FF}$ be 
the normal bundle to $V$ in $\FF:=\FF(a_0,\dots,a_4)$. For simplicity, we also 
put 
$\EEE:=\oplus \OOO_{\PP^1} (a_i)$. From the exact sequence
\[
0 
\longrightarrow\mathscr{T}_{V}\longrightarrow\mathscr{T}_{\FF}
\longrightarrow\NNN_{V/\FF}\longrightarrow 0
\]
we obtain the following relation on Chern classes:
\[
\cc_{t}\left (\mathscr{T}_{V}\right)= \cc_{t}\left (\mathscr{T}_{\FF}\right) 
\cdot \cc_{t}\bigl (\NNN_{V/\FF}\bigr)^{-1}.
\]
From relative Euler exact sequence
\[
0 \longrightarrow \OOO_{\FF} \longrightarrow f^* \EEE^\vee \otimes 
\OOO_{\FF}(M) \longrightarrow \mathscr{T}_{\FF/\PP^1} \longrightarrow 0
\]
we obtain $\cc_{t}\left (\mathscr{T}_{\FF/\PP^1}\right)=\cc_{t}\left (f^* 
\EEE^\vee \otimes \OOO_{\FF}(M)\right) $. Hence,
\begin{multline*}
\cc_{t}\left (\mathscr{T}_{V}\right)=
\cc_{t}(f^*\mathscr{T}_{\PP^1}) \cdot \cc_{t}\left (f^* \EEE^\vee \otimes 
\OOO_{\FF}(M)\right) \cdot \cc_{t}\bigl (\NNN_{V/\FF}\bigr)^{-1}=
\\
=(1+2Ft)\cdot \prod_i \left (1+(M-a_iF)t\right) \prod_j \left (1+(2M+b_jF)t\right)^{-1}.
\end{multline*}
Taking into account that $\Eu(V)=\mathrm{c}_3(\mathscr{T}_{V})$, 
$M^4\cdot F=1$ and $M^5=\sum a_i$ we obtain the desired equality.
\end{proof}

\begin{proposition}
All the varieties described in 
\ref{blowup-V14}, 
\ref{conic-conic},
\ref{blowup-cubic-point},
\ref{conic-blowup}, 
\ref{g=6:nonG}, 
\ref{dP-blowup}
are not rational.
General varieties described in 
\ref{dP-dP}, 
\ref{blowup-W2-point}, 
\ref{g=5:trig},
\ref{g=5:conic-G}, \ref{dP-conic}, \ref{g=6:trig}
are not rational.
\end{proposition}

\begin{proof}
Let $X$ be the variety of type~\ref{dP-dP} and let $\hat X$ be its $\QQ$-factorialization.
Then $\hat X$ can be realized in $\FF[0^3, 1^2]$ as an intersection of two divisors $D_1,\, D_2\in |2M|$ 
according to~\ref{ex:dP-dP}. Assume that  $\hat X$ is smooth.
Then by Lemma~\ref{lemma:Eu} and Theorem~\ref{thm:DPal} the variety $\hat X$ not rational. 

By the description in Theorem~\ref{th:rho=2} any variety of type 
\ref{blowup-cubic-point}, 
\ref{g=6:nonG}, or~\ref{dP-blowup}
is birationally equivalent to a smooth cubic in $\PP^4$. It is well known 
that it is nonrational \cite{Clemens-Griffiths}.
The same holds for type~\ref{conic-conic} (see~\ref{ssect:conic-conic}).
Similarly, any variety of type~\ref{blowup-V14} or~\ref{conic-blowup} is birationally equivalent 
to a smooth Fano threefold $Y=Y_{14}\subset \PP^9$ of genus $8$ (with $\uprho(Y')=\iota(Y')=1$). Such a threefold is again birationally equivalent to 
a smooth cubic \cite[Theorem~4.5.8]{IP99} (see also \cite{Iskovskikh1980}, \cite{Tregub1985a}, \cite{Takeuchi-1989}) and hence it is nonrational \cite{Clemens-Griffiths}.

By the description in Theorem~\ref{th:rho=2} a
general variety of type~\ref{blowup-W2-point} or~\ref{g=6:trig} is birationally 
equivalent to a smooth double space branched in a quartic. It is also 
nonrational \cite{Voisin1988} (see 
\cite{Prz-Ch-Shr:DS} for special cases).

Any variety of type~\ref{g=5:trig},
\ref{g=5:conic-G}, or~\ref{dP-conic}
has a structure of conic bundle $f:\hat X\to \PP^2$ with 
discriminant curve $\Delta\subset \PP^2$ of degree $\ge 6$. 
If the $\QQ$-factorialization $\hat X$ is smooth, then $f$ is a standard 
conic bundle and $X$ is nonrational 
by Corollary~\ref{cb:P2-0}.
\end{proof}

\section{Proof of Theorem~\ref{th:rho=2}} 
\label{sect:pfTh12}

\subsection{Beginning of the proof}
The existence of a link \eqref{eq:diagram} follows from standard 
arguments (see, e.g., \cite[\S 4.1]{IP99}).
Indeed, since $\uprho(X)=1$ and $\rk\Cl(X)=2$, the variety $X$ has two 
small $\QQ$-factorializations 
$\tau:\hat X\to X$ and $\tau':\hat X'\to X$ (see \cite[Corollary 
4.5]{Kawamata:crep}) and they are connected by a flop $\chi$. The 
divisors $-K_{\hat X}=\tau^*(-K_X)$ and 
$-K_{\hat X'}=\tau^{\prime*}(-K_X)$ are nef and big. Hence there exist Mori 
contractions $f: \hat X\to Y$ 
and $f': \hat X'\to Y'$. 
They are described by \cite{Mori:3-folds}, \cite{Cutkosky:contr} (see 
\cite[Corollary~2.5]{P:ratF-1}).
Since the threefold $X$ is not rational, the contractions $f$ and $f'$ cannot 
be neither of types \type{C_2}, \type{D_2} nor \type{D_3}
(see e.g. Corollary~\ref{cb:P2} and Lemma~\ref{lemma:del-pezzo}).
Hence they are of types \type{E_1}-\type{E_5}, \type{C_1} 
or \type{D_1}.
Note that $X$ is rationally connected, so the varieties $Y$ and $Y'$ are. 
Hence $Y$ must be a smooth rational curve in the case where $f$ is of type \type{D_1}. If $f$ is a contraction of type \type{C_1}, then $Y$ is a smooth rational surface \cite{Cutkosky:contr}. Since $\uprho(Y)=1$, $Y\simeq \PP^2$
in this case. If $f$ is a contraction of types \type{E_1}-\type{E_5},
then $Y$ is a (possibly singular) Fano threefold with $\uprho(Y)=1$.

\subsection{Notation}
Denote $H:=-K_X$.
Then $\bil{H}{H}=(-K_X)^3=2g-2$, where $g=\g(X)$ is the genus of $X$. 
Let $H_Y$ and $H_Y'$ be the ample generators 
$\Pic(Y)$ and $\Pic(Y')$, respectively. If the contraction $f$ (resp., $f'$) is 
birational, then by 
$E$ (resp., by $E'$) we denote the exceptional divisor. Put 
\[
F:=f^*H_Y, \quad
F':=f'^*H_Y', \quad
M:= \tau_*F, \quad
M':=\tau'_*F', \quad
D:=\tau_*E, \quad
D':=\tau_*E'.
\]

For any threefold $V$ with terminal Gorenstein 
singularities one can define the following integral bilinear form on 
on $\Cl(X)$:
\[
\bil{D_1}{D_2}:=(-K_V)\cdot D_1\cdot D_2.
\]
The value of this form is preserved under flops, i.e. if 
$\chi: V\dashrightarrow V'$ is a flop, then 
\[
\bil{D_1}{D_2}=\bil{\chi_*D_1}{\chi_*D_2}.
\]

Consider the possibilities for the contraction $f$ case by case.

\subsection{Type \type{B_2}}
Then $f$ is the blowup of a smooth point.
Then $Y$ is a locally factorial Fano threefold with $\uprho(Y)=1$ and 
$(-K_X)^3=(-K_{\hat X})^3=(-K_Y)^3-8$.
Hence, $(-K_Y)^3\ge 16$.
According to Theorem~\ref{th:rho=1} we have 
$\iota(Y)=2$ and $\dd(Y)=2$ or 
$3$.
We obtain cases~\ref{blowup-cubic-point} and~\ref{blowup-W2-point}.

\subsection{Types \type{B_{3-4}}}
Then $f$ is the blowup of a singular (Gorenstein) point of type 
\type{cA_1}.
Then, as above, $Y$ is a locally factorial (singular) Fano threefold with $\uprho(Y)=1$ 
and 
\[
8\le (-K_X)^3=(-K_{\hat X})^3=(-K_Y)^3-2.
\]
According to Theorem~\ref{th:rho=1} for $Y$ there are only two possibilities: 
$\iota(Y)=\dd(Y)=2$ or $\iota(Y)=1$ and $\g(Y)=6$.
However, in the former case the anticanonical linear system $|-K_{\tilde X}|$ 
does not contract any curves (see 
\cite[Theorem~2]{P:factorial-Fano:e}), $\tilde X\simeq X$ and so $\uprho(X)=2$, a contradiction.
Hence, $\iota(Y)=1$ and $\g(Y)=6$. The corresponding link is described in \cite[Theorem~2]{P:factorial-Fano:e}. We obtain the case~\ref{g=5:conic-G}.

\subsection{Type \type{B_1}}
Then $f$ contracts a divisor $E\subset \hat X$ to a curve $\Gamma\subset Y$.
Thus $Y$ is a locally factorial Fano threefold with $\uprho(Y)=1$.
Moreover, 
\begin{equation}
\label{eq:KKB1}
\begin{array}{rcl}
-K_{X}^3 &=&(-K_Y)^3 - 2(-K_Y) \cdot \Gamma+2\p(\Gamma)-2,
\\[0.5em]
(-K_{\hat X})^2\cdot E &=& (-K_Y) \cdot \Gamma +2- 2\p(\Gamma)
\end{array}
\end{equation} 
(see e.g. \cite[Proposition~5.1]{P:planes}).

Consider the case $\iota(Y)\ge 2$. According to Theorem~\ref{th:rho=1} we 
have 
\[
\iota(Y)=2\quad \text{ and}\quad \dd(Y)\le 3.
\]
In this case the numbers $-K_Y \cdot \Gamma$ and $(-K_{\hat X})^2\cdot E$ are 
even (and positive).
Thus we can write
\[
-K_Y \cdot \Gamma=2m\quad \text{and}\quad (-K_{\hat X})^2\cdot E=2n. 
\]
for some positive integers $m$ and $n$. Then \eqref{eq:KKB1} can be rewritten as follows
\begin{eqnarray*}
g+2m&=&4\dd(Y) +\p(\Gamma),
\\
n+ \p(\Gamma) &=& m +1.
\end{eqnarray*}
Elimination $\p(\Gamma)$ from the first equation and $m$ from the second one 
we obtain
\begin{equation}
\label{eq:rho=2:iota=2:B1}
\begin{array}{rcl}
g+n+m&=&4\dd(Y) +1,
\\
g+2n+ \p(\Gamma)&=&4\dd(Y) +2.
\end{array}
\end{equation}
In particular, $\dd(Y)>1$ because $g+n+m\ge 7$. 

Assume that $\dd(Y)=2$. 
Then $m\le 3$ by \eqref{eq:rho=2:iota=2:B1}. If 
$m=1$, then $\Gamma$ is a line on $Y$ and so we obtain the case~\ref{g=6:trig}
(see \cite[Example~4.7]{P:ratF-1}).
If $m=3$, then $n=1$, $g=5$, $\p(\Gamma)=3$.
Recall that the linear system $|-\frac12K_Y|$ is base point free and defines a 
double cover $\pi: Y\to\PP^{3}$ \cite[Corollary 0.8]{Shin1989}.
In our case the restriction $\pi_{\Gamma}: \Gamma\to \pi(\Gamma)$ of $\pi$ to 
$\Gamma$ is birational onto its image, 
hence $\p(\pi(\Gamma))\ge \p(\Gamma)=3=\deg \pi(\Gamma)$. It is clear that this 
is impossible. 
Hence, $m=2$. Now we use the fact that the
linear system $|-K_Y|$ is very ample and defines an embedding $Y\subset\PP^{10}$ 
\cite[Corollary 0.8]{Shin1989}
so that the image of $\Gamma$ is a curve of degree $4$.
This implies that $\p(\Gamma)\le 1$ and so $g=5$, $n=2$, $\p(\Gamma)= 1$ by 
\eqref{eq:rho=2:iota=2:B1}.
In this case the morphism $\pi_{\Gamma}$ must be a double cover, $\pi(\Gamma)$ 
is a line in $\PP^3$
and $\Gamma$ is a intersection of 
two elements of the linear system $|-\frac 12 K_{Y}|$.
But then the divisor $-K_{\hat X}$ is ample and so $\uprho(X)=2$, a 
contradiction. 

Assume that $\dd(Y)=3$. 
Then $Y=Y_3\subset\PP^4$ is a smooth cubic and $\Gamma$ is a curve of degree 
$m$.
From \eqref{eq:rho=2:iota=2:B1} we have $m\le 7$.
If $\langle\Gamma\rangle\cap X=\Gamma$, then the linear system $|F-E|$ is base 
point free 
and so $-K_{\hat X}=2F-E$ is ample. This contradicts our assumption.
From now on we assume that $\langle\Gamma\rangle\cap X\neq \Gamma$.
If $\dim \langle\Gamma\rangle=2$, then $\Gamma$ must be a plane irreducible 
conic.
We obtain the case~\ref{dP-blowup}.
Assume that $\langle\Gamma\rangle=\PP^4$. Then $m\ge 4$.
Applying Castelnuovo's bound of the genus of curves in projective space 
\cite[Ch.~III, \S~2]{A-C-G-H}
we see that $\p(\Gamma)\le m-4$. Then \eqref{eq:rho=2:iota=2:B1} has no 
solutions.
Thus $\dim \langle\Gamma\rangle= 3$, i.e. $\Gamma$ is contained in a hyperplane.
Hence $|F-E|\neq \varnothing$ and 
\begin{equation}
\label{eq:rho=2:iota=2:B1-a}
0<(-K_{\hat Y})^2\cdot (F-E)=\frac12 (-K_{\hat Y})^2\cdot (-K_{\hat Y}-E)=g-1-n
\end{equation} 
(because $\tau$ does not contract any divisors).
Again apply Castelnuovo's bound \cite[Ch.~III, \S~2]{A-C-G-H}. 
For $m=5$, $6$ and $7$ we obtain that $\p(\Gamma)$ is less or equal to $2$, $4$ 
and $6$, respectively.
Then \eqref{eq:rho=2:iota=2:B1} has no solutions.
If $m=3$, then $\p(\Gamma)=0$ and we obtain the case~\ref{g=6:nonG}.
Finally, if $m=4$, then $\p(\Gamma)=1$, $n=4$, and $g=5$.
This contradicts \eqref{eq:rho=2:iota=2:B1-a}.

Now we assume that $\iota(Y)=1$. Put $-K_Y \cdot \Gamma=m$ and $(-K_{\hat 
X})^2\cdot E=n$. Then \eqref{eq:KKB1} can be rewritten as follows
\begin{equation}
\label{eq:rho=2:iota=2:B1-iota1}
\begin{array}{rcl}
g&=&\g(Y)-1 - m+\p(\Gamma),
\\
n &=& m +2- 2\p(\Gamma).
\end{array}
\end{equation} 
Since $\g(Y)\le 8$, $\g(Y)\neq 7$, and $g\ge 5$, this implies $m\le 5$ and 
$\g(Y)=g+ (n +m)/2>g$.
Then $-K_Y$ is very ample and the anticanonical image of $Y$ is an intersection 
of quadrics
(see e.g. \cite[Theorem~1.1]{P:ratF-1}).
Then $\p(\Gamma)\le 1$. 
Using this and \eqref{eq:rho=2:iota=2:B1-iota1} we deduce that $m\le 3$.
Then obviously, $\p(\Gamma)=0$ and again from \eqref{eq:rho=2:iota=2:B1-iota1} 
we get
two possibilities for $(g,\g(Y),n, m,\p(\Gamma))$:
$(5, 8, 4, 2, 0)$ and $(6, 8, 3, 1, 0)$.
These are the cases~\ref{blowup-V14} and~\ref{conic-blowup}.

\subsection{Type \type{B_5}}
\label{rho=2:cb-B5}
Then the contraction
$f: X\to Y$ is birational and contracts a divisor $E$ to a non-Gorenstein 
singular point $P\in Y$. 
In this case $Y$ is a Fano threefold with $\QQ$-factorial
terminal singularities and $\uprho(Y)=1$, $E\simeq \PP^2$ and 
$\OOO_E(E)\simeq\OOO_{\PP^2}(-2)$. 
Hence we have
(see, e.g., \cite[Lemma~4.1.6]{IP99}):
\begin{equation}
\bil{D}{D}=-2,\quad \bil{D}{H}=1. 
\end{equation}
By the above we may assume that the contraction $f'$ is not of types 
\type{B_1}-\type{B_4}.
Since 
\[
\Cl(X)=\ZZ\cdot H\oplus \ZZ\cdot D=\ZZ\cdot H\oplus \ZZ\cdot M', 
\]
we can write
$M'\sim a H- D$. Then
\[
\bil{M'}{M'}=a^2(2g-2)-2a-2\ge 4.
\]
On the other hand, if $f'$ is a del Pezzo fibration (resp. conic bundle), then 
$\bil{M'}{M'}=0$ (resp. $2$), a contradiction.
Finally, if $f'$ is of type \type{B_5}, then similarly $D'\sim a H- D$ and
\[
-2= \bil{D'}{D'}=a^2(2g-2)-2a-2\ge 4,
\]
a contradiction.

\subsection{}
Thus from now on we may assume that both contractions$f$ and $f'$ 
are not birational.
Below in subsections~\ref{rho=2:dP-dP} and~\ref{rho=2:dP-cb} we 
consider cases where $f$ is a del Pezzo fibration.
We have the following relations:
\begin{equation}
\label{rho=2:eq:dP}
\bil{M}{M}=0,\quad \bil{H}{M}=d,
\end{equation}
where $d$ is the degree of the generic fiber of $f$.
By Lemma~\ref{lemma:del-pezzo} $d\le 4$.

\subsection{Types \type{D_1}-\type{D_1}}
\label{rho=2:dP-dP}
Let both contractions $f$ and $f'$ be del Pezzo fibrations.
Put $d:=K_{F}^2$ and $d':=K_{F'}^2$.
Similarly to \eqref{rho=2:eq:dP} we have the following relations:
\begin{equation*}
\bil{M'}{M'}=0, \quad \bil{H}{M'}=d'. 
\end{equation*}
Since $\Cl(X)=\ZZ\cdot H\oplus \ZZ\cdot D=\ZZ\cdot H\oplus \ZZ\cdot D'$, we 
can write
$M'\sim a H- M$.
Hence, 
\[
(2g-2) a^2-2ad=0,\quad (2g-2) a-d=d'.
\]
Hence, $(g-1) a=d=d'$, $a=1$ and $g=5$.
We obtain the case~\ref{dP-dP}.

\subsection{Types \type{D_1}-\type{C_1}}
\label{rho=2:dP-cb}
Consider 
the case where $f$ is a del Pezzo fibration and $f'$ is a conic bundle. Put 
$d:=K_{F}^2$.
Then $Y'\simeq\PP^2$. Let $\Delta'\subset Y'$ be the discriminant curve and let 
$d':=\deg \Delta'$.
We have the following relations:
\begin{equation*}
\bil{M}{M}=0, \quad \bil{M'}{M'}=2, \quad \bil{H}{M}=d, \quad 
\bil{H}{M'}=12-d'. 
\end{equation*}
Since $\Cl(X)=\ZZ\cdot H\oplus \ZZ\cdot M=\ZZ\cdot H\oplus \ZZ\cdot M'$, we 
can write
$M'\sim a H- M$. Hence, 
\[
(2g-2) a^2-2ad=2,\quad (2g-2) a-d=12-d'.
\]
This gives us
\[
d=g-2,\quad 2g-2=12-d'+d=12-d'+g-2,\quad
g=12-d'.
\]
where $d'\ge 5$ by Corollary~\ref{cb:P2}.
Then $a=1$, $d=g-2$ and $d'=12-g$, where $g\ge 5$ and $d\le 4$. We obtain cases 
\ref{dP-conic} and~\ref{g=5:trig}.

\subsection{Types \type{C_1}-\type{C_1}}
\label{rho=2:cb-cb}
Finally, we consider the case where both contractions $f$ and $f'$ are
conic bundles.
Then $Y\simeq Y'\simeq\PP^2$. Let $\Delta\subset Y$ and $\Delta'\subset Y'$ be 
the corresponding discriminant curves, let $d:=\deg \Delta$ 
and let $d':=\deg \Delta'$.
We have the following relations (see, e.g., \cite[Lemma~4.1.6]{IP99}):
\begin{eqnarray*}
\bil{M}{M}&=&2,\quad \bil{H}{M}=12-d,
\\
\bil{M'}{M'}&=&2,\quad \bil{M'}{H}=12-d'. 
\end{eqnarray*}
Since $\Cl(X)=\ZZ\cdot H\oplus \ZZ\cdot M=\ZZ\cdot H\oplus \ZZ\cdot M'$, we 
can write
$M'\sim a H- M$. Hence, 
\[
(2g-2) a^2-2(12-d)a+2=2,\quad (2g-2) a-(12-d)=12-d'.
\]
By Corollary~\ref{cb:P2} we have $d, d'\ge 5$.
Taking into account that $g\in \{5,\dots, 10,12\}$ we 
immediately obtain $a=1$ and
$d=d'=13-g$, i.e. 
\begin{equation}
\label{eq:KMM}
-K_X=M+M'
\end{equation} 
and 
\[
\deg \Delta=\deg\Delta'=13-g,\qquad g\in \{8, 7, 6, 5\}. 
\]
Further, put $F^{\sharp}:=\chi^{-1}_* F'=\tau^{-1}_*M'$. 
From \eqref{eq:KMM}
we obtain $-K_{\hat X}=F+F^{\sharp}$ and 
\begin{multline*}
(F^{\sharp})^3= (-K_{\hat X} -F)^3= (-K_{\hat X})^3- 3(-K_{\hat X})^2\cdot F +3(-K_{\hat X})\cdot F^2 =
\\
= (-K_{X})^3- 3\bil{H}{M}+3\bil{M}{M}
=2g-2 -3(12-d)+6=7-g.
\end{multline*}
Take two general members $F_i'\in |F'|$ and let 
$F_i^{\sharp}:=\chi^{-1}_*(F_i')$. Then $l':=F_1'\cap F_2'$ is an irreducible
rational curve which is a fiber $f'$.
This curve does not meet the flopped locus of $\hat X'$ and, therefore, 
its proper transform 
$l^{\sharp}:=\chi^{-1}(l')$ is an irreducible rational curve which does not meet the 
flopped locus of $\hat X$. 
Note that the divisors $F_j^{\sharp}$ are $\tau$-negative, i.e. 
$F_j^{\sharp}\cdot l_i<0$ for all $i$.
Hence $F_1^{\sharp}\cap F_2^{\sharp}= l^{\sharp}+\sum a_i l_i$, where $l_i$ is the 
flopped curves and $a_i>0$.
The curves $l^{\sharp}$ swept out an open subset in $\hat X$. 
Further, $F^{\sharp}\cdot l^{\sharp}=0$. Hence,
\[
7-g=(F^{\sharp})^3= F^{\sharp}\cdot l^{\sharp}+\sum a_i F^{\sharp}\cdot l_i=\sum 
a_i F^{\sharp}\cdot l_i<0, \quad g>7.
\]
Thus $g=8$. 
Here the flopped locus is an irreducible curve $l$ such that
$F^{\sharp}\cdot l=-1$, $F\cdot l=1$ and $f(l)$ is a line on $Y\simeq \PP^2$.
We obtain the case~\ref{conic-conic}.

This completes the proof of Theorem~\ref{th:rho=2}.

\section{Proof of Theorem~\ref{th:rho=3}} 
\label{sect:pfTh14}
In this section we prove Theorem~\ref{th:rho=3}.
\begin{ass}
\label{ass:rho=3}
Let $X$ a Fano threefold with at worst terminal Gorenstein 
singularities with $\rk\Cl(X)\ge 3$ and $g:=\g(X)\ge 5$. 
Assume that $X$ is not rational and for any small $\QQ$-factorialization $\tau: 
\hat 
X\to X$, any birational contraction of $\hat X$ is small.
\end{ass}

\begin{lemma}
\label{claim1}
Let $\tau: \hat X\to X$ be a small $\QQ$-factorialization.
Then any extremal Mori contraction on $\hat X$ is a conic bundle 
over $\PP^1\times \PP^1$. In particular, $\uprho(\hat X)=3$.
Let $\Delta\subset \PP^1\times \PP^1$ be the discriminant curve. Then $\Delta$ is a divisor of type 
$(n_1,n_2)$, where $n_1,\, n_2\ge 4$.
\end{lemma}

\begin{proof}
Let $\pi: \hat X\to Z$ be an extremal Mori contraction.
By our assumption it is not birational.
Since $\uprho(Z)=\uprho(\hat X)-1\ge 2$, $\pi$ is not a del Pezzo fibration.
Therefore, $\pi$ is a conic bundle and $Z$ is a smooth surface \cite{Cutkosky:contr}.
Moreover, $Z\simeq \PP^1\times\PP^1$ by \cite[Theorem~3.1]{P:ratF-1}.

Let $l_i$, $i=1,\, 2$ be general rulings of $Z=\PP^1\times\PP^1$. 
Put $n_i:=\Delta\cdot l_i$.
Let $\hat{S}_i:=\pi^{-1}(l_i)$. 
By Bertini's theorem $\hat{S}_i$ is a smooth surface and it 
does not contain $K_{\hat X}$-trivial curves. 
Hence, $-K_{\hat{S}_i}=-K_{\hat X}|_{\hat{S}_i}$ is ample, i.e. $\hat{S}_i$ is a del Pezzo surface.

On the other hand, $\pi_{\hat{S}_i}: \hat{S}_i\to l_i$ is a conic bundle whose degenerate fibers 
lie over the points $\Delta\cap l_i$. Hence $K_{\hat{S}_i}^2=8-n_i$ and so 
\[
8-n_i= K_{\hat{S}_i}^2= (K_{\hat X}+\hat{S}_i)^2\cdot \hat{S}_i = (-K_{\hat X})^2\cdot 
\hat{S}_i.
\]
Since $\hat X$ is not rational, $n_i\ge 4$ by 
Proposition~\ref{prop:conic-bundle:rat}.
\end{proof}

\begin{lemma}
\label{claim2}
We have $\iota(X)=1$ and 
the anticanonical divisor $-K_X$ is very ample.
\end{lemma}

\begin{proof}
By Lemma~\ref{claim1} we have $\Delta\neq \varnothing$. Hence $-K_{\hat X}$ is a primitive element of $\Pic(\hat X)$, so $\iota(X)=1$.
So, the assertion follows from Proposition~\ref{prop:v-ample}.
\end{proof}

\begin{slemma}
\label{claim1a}
In the notation of Lemma~\xref{claim1} one has $n_1,\, n_2\le 5$
\end{slemma}
\begin{proof}
The divisor $-K_{X}$ is very ample by Lemma~\ref{claim2}, hence $-K_{\hat{S}_i}$ is very ample as well.
Therefore, $8-n_i=K_{\hat{S}_i}^2\ge 3$. 
\end{proof}

\begin{slemma}
\label{claim-Eff}
The cone of effective divisors $\Eff(X)\subset \mathrm{N}^1(X)=\RR^3$ 
polyhedral and generated by the classes of movable divisors $S_1,\dots, S_r$, $r\ge 3$ such that $\dim |S_i|=1$. For each $S_i$ there exists a small 
$\QQ$-factorialization $\tau: \hat X\to X$ such that the proper transform 
$|\hat S_i|$ of $|S_i|$ is base point free and defines a \textup(non-minimal\textup) del Pezzo fibration $\varphi_i: \hat X\to \PP^1$. Each two-dimensional face $\mathrm{F}\subset \Eff(X)$ is
generated, up to permutations of $S_1,\dots, S_r$, by the classes of two divisors $S_i$ and $S_{i+1}$ \textup(the subscript indices of $S_i$ are considered modulo $r$\textup).
\end{slemma}

\begin{proof}
For any small $\QQ$-factorialization $\pi: \hat X\to X$ we have a natural identification $\Eff(X)=\Eff(\hat X)$ as well as an inclusion 
$\Nef(\hat X)\subset \Eff(\hat X)$. Moreover, since there are no divisorial contractions on $\hat X$, for any element $D\in \Eff(\hat X)$ there is a finite sequence of flops $\chi : \hat X\dashrightarrow \hat X'$ over $X$ such that $\chi_*D$ is nef on $\hat X'$. Therefore, there exists a chamber decomposition 
\[
\Eff(X)=\bigcup_{\hat X} \Nef(\hat X),
\]
where the union is finite and runs through all small $\QQ$-factorializations $\pi: \hat X\to X$.
Note that $\hat X$ is a $\QQ$-factorial FT variety (i.e. it is a klt log Fano threefold with respect to some boundary). By the cone theorem 
the Mori cone $\NE(\hat X)$ is generated by a finite 
number of extremal rays. This implies that 
the same is true for its dual cone $\Nef(\hat X)$ 
and the extremal rays of 
$\Nef(\hat X)$ are generated by a finite number of integral divisors.
Therefore, the cone $\Eff(X)$ is generated by integral divisors
$S_1, S_2,\dots$. 
By Kodaira's lemma \cite[Lemma~0-3-3]{KMM} the boundary of $\Eff(X)$ contains no big elements. 
Thus for every $S_i$ there is a $\QQ$-factorialization $\pi_i: \hat X_i\to X$ such that the proper transform $\hat S_i$ is nef.
Since $\hat X_i$ is a weak Fano threefold, the linear system $|\hat S_i|$ 
defines a contraction $\varphi_i: \hat X_i\to Y_i$ 
to a lower dimensional variety.
This contraction passes through a Mori contraction:
\[
\varphi_i: \hat X_i \overset{\psi_i}\longrightarrow Z_i \longrightarrow Y_i.
\]
By Lemma~\ref{claim1} the contraction $\psi_i$ is a conic bundle and $Z_i\simeq \PP^1\times \PP^1$.
Hence $Y_i\simeq \PP^1$ and $\hat S_i$ is a del Pezzo surface.
Then the degree of $S_i$ is bounded:
\[
(-K_X)^2\cdot S_i=(-K_{\hat X_i})^2\cdot \hat S_i\le 9.
\]
This implies that modulo algebraic equivalence there is only a finite number of 
divisors $S_i$.
\end{proof}

Thus, the transversal section of the effective cone $\Eff(X)$ has the following form:
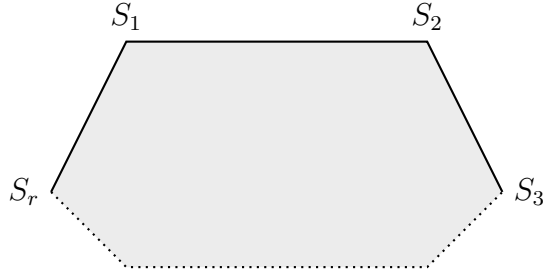
\[
\begin{tikzpicture}[scale=1]
\path[draw, black, thick,fill=gray!15] 
(-3,0) node[black, left]{$S_r$} --
(-2,2) node[black, above]{$S_1$} -- 
(2,2) node[black, above]{$S_2$} -- 
(3,0) node[black, right]{$S_3$} ;
\path[draw, black, thick,fill=gray!15,dotted] 
(-3,0) --(-2,-1)--(2,-1) -- (3,0)
;
\end{tikzpicture} 
\]

\begin{lemma}
\label{claim4}
For all $i$ we have $K_{S_i}^2=4$ and one of the following holds:
\begin{enumerate}
\item \label{claim4-caseg=5}
$g=5$,\, $r=4$,\, $S_{1}+S_{3}\sim S_{2}+S_{4}\sim -K_{X}$;
\item \label{claim4-caseg=7}
$g=7$,\, $r=3$,\, $S_{1}+S_{2}+S_{3}\sim -K_{X}$.
\end{enumerate} 
\end{lemma}

\begin{proof}
Since the group $\Cl(X)$ generated by the classes of divisors $-K_{X}$, $S_i$ and 
$S_{i+1}$, we can write 
\begin{equation}
\label{eq:rho=3:KF}
S_{i+2}=a_{i}(-K_{X})+b_{i}S_{i}+ c_iS_{i+1},\qquad a_i,\, b_i,\, c_i\in \ZZ. 
\end{equation} 
Similarly, the class of $S_i$ can be expressed in terms of $-K_{X}$, $S_{i+1}$ and $S_{i+2}$. Hence, 
$b_i=\pm 1$. Since $S_{i+2}$ is effective and its class is not contained in 
the face generated by $S_i$ and $S_{i+1}$, we have $a_i>0$.
Further, 
\[
-K_X\cdot S_{i+2}\cdot S_{i+1}=2= 4a_i+2b_i,\qquad b_i=1-2a_i.
\]
Hence, $b_i=-1$, $a_i=1$, i.e. 
\begin{equation}
\label{eq:rho=3:KF-a}
S_{i+2}=-K_{X}-S_{i}+ c_iS_{i+1},\qquad c_i\in \ZZ. 
\end{equation} 
Similarly, we have
\begin{equation}
\label{eq:rho=3:KF-b}
S_{i+3}=-K_{X}-S_{i+1}+ c_{i+1}S_{i+2}. 
\end{equation}
Combining this with \eqref{eq:rho=3:KF-a} we obtain
\begin{equation}
\label{eq:rho=3:KF-c}
S_{i+3}=(1+c_{i+1})(-K_{X})-c_{i+1}S_{i}+ (c_ic_{i+1}-1)S_{i+1}. 
\end{equation}
Since $S_{i+3}$ effective, $c_{i+1}\ge -1$ for all $i$.
Denote 
\[
d_i:= K_{S_i}^2= (-K_{X})^2\cdot S_i.
\]
From \eqref{eq:rho=3:KF-a} we have
\begin{equation}
\label{eq:rho=3:KF-d}
d_{i+2}+d_i=2g-2+ c_id_{i+1}. 
\end{equation} 
Since $g\ge 5$ and $3\le d_i\le 4$ for all $i$, we have $c_i\le 0$.
Moreover, $c_i=0$ for some $i$ if and only if $g=5$, and in this case $c_i=0$
and $d_i=4$ for all $i$. 
Then \eqref{eq:rho=3:KF-a} has the form
\[
-K_{X}= S_{i+2}+S_{i}= S_{i+3}+S_{i+1}= S_{i+4}+S_{i+2}.
\]
In particular, $S_{i+4}=S_{i}$, i.e. $r=4$.
We obtain the case~\ref{claim4-caseg=5}.

Assume that $c_i=-1$ for all $i$. Then $g>5$. From \eqref{eq:rho=3:KF-c} we obtain $S_{i+3}=S_i$, i.e. $r=3$. 
This is the case~\ref{claim4-caseg=7}.
\end{proof}

\begin{lemma}
\label{claim5}
The case \xref{claim4}\xref{claim4-caseg=7} does not occur.
\end{lemma}

\begin{proof}
Assume that $g=7$. Let $\hat S_i$ be the proper transform $S_i$ on 
$\hat X$. Then $-K_{\hat X}=\hat S_1+\hat S_2+\hat S_3$. 
Since $\bil{S_i}{S_j}=2\updelta_{i,j}$, we have
\[
\begin{array}{rcccllcl}
(-K_{\hat X}) \cdot \hat S_1\cdot \hat S_2 &=&\bil {S_1}{S_2}= &2&=& \hat S_3 \cdot 
\hat S_1\cdot \hat S_2,
\\
(-K_{\hat X}) \cdot \hat S_i \cdot \hat S_3 &=&\bil {S_i}{S_3}=&2&=& \hat S_3 
\cdot \hat S_1\cdot \hat S_2 +\hat S_3^2 \cdot \hat S_i&=& 2 +\hat S_3^2 \cdot 
\hat S_i.
\end{array}
\]
Hence, $\hat S_3^2 \cdot \hat S_i=0$. This shows that the pencil
$|\hat S_3|$ is base point free. 
On the other hand, $\hat S_3$ must be negative on some flipping curve, a 
contradiction.
\end{proof}

\begin{remark}
The variety of type \xref{claim4}\xref{claim4-caseg=7} occurs if we relax the condition $\uprho(X)=1$: `a smooth double cover of 
$\PP^1\times \PP^1\times \PP^1$ whose branch divisor is a member of $|-K_{\PP^1\times \PP^1\times \PP^1}|$ is a nonrational Fano threefold 
with $\uprho(X)=3$ and $\g(X)=7$. 
\end{remark}

Let us summarize the facts obtained above: 

\begin{proposition}
\label{ex:r=3}
Let $X$ be a Fano threefold satisfying Assumption~\xref{ass:rho=3}.
Then the following assertions hold.
\begin{enumerate}
\item 
\label{ex:r=3a}
$\g(X)=5$, $\rk \Cl (X)=3$, the effective cone $\Eff(X)$ is polyhedral, generated by effective prime divisors $S_1,\dots, S_4$ such that 
\[
-K_X\sim S_1+S_3\sim S_2+S_4, (-K_X)^2\cdot S_i=4 
\]
and . The transversal section of $\Eff(X)$ has the following form:
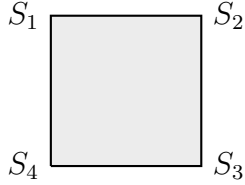
\[
\begin{tikzpicture}[scale=1]
\path[draw, black, thick,fill=gray!15] 
(-1,0) node[black, thick,left]{$S_4$} --
(-1,2) node[black,thick, left]{$S_1$} -- 
(1,2) node[black,thick, right]{$S_2$} -- 
(1,0) node[black, thick,right]{$S_3$}-- 
(-1,0);
\end{tikzpicture} 
\]
\item 
\label{ex:r=3b}
For every $S_i$ there is a $\QQ$-factorialization $\pi_i: \hat X_i\to X$ such that the proper transform $\hat S_i$ is defines a \textup(non-minimal\textup) del Pezzo fibration 
$\varphi_i: \hat X_i\to \PP^1$ which passes through two conic bundles $\psi_i'$ and $\psi_i''$ with discriminant curves of bidegree $(4,4)$:
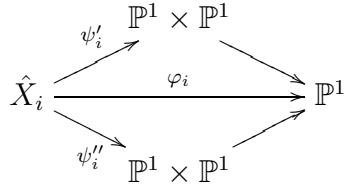
\[
\xymatrix@R=0.8em{
&\PP^1\times \PP^1\ar[rd]&
\\
\hat X_i\ar[ru]^{\psi_i'}\ar[rd]_{\psi_i''}\ar[rr]^{\varphi_i} && \PP^1
\\
&\PP^1\times \PP^1\ar[ru]&
}
\]
\item 
\label{ex:r=3c}
The anticanonical image $X=X_8\subset \PP^6$ 
is a complete intersection of three quadrics so that two of them have corank 
$3$. The singular locus of a general variety from this family consists of $8$ nodes.
Conversely, a general complete intersection of three quadrics $Q_1\cap Q_2\cap 
Q_3\subset \PP^6$, 
where $\corank(Q_1)=\corank(Q_2)=3$, is a threefold satisfying Assumption~\xref{ass:rho=3}.
\end{enumerate}
\end{proposition}

\begin{proof}
The assertions of~\ref{ex:r=3a} and~\ref{ex:r=3b} are consequences of
Lemmas~\ref{claim1}--\ref{claim5}. Let us prove~\ref{ex:r=3c}.
Similarly to Proposition~\ref{claim:case1-0}
the threefold $X=X_8\subset \PP^6$ is a complete intersection of three quadrics:
$X=Q_1\cap Q_2\cap Q_3\subset \PP^6$ and for a general member $\hat S_i\in |\hat S_i|$ the image $S_i=\pi_i(\hat S_i)\subset X\subset \PP^6$ is a smooth quartic del Pezzo surface. There is a quadric in the net containing the linear span $\langle S_i\rangle$. We may assume that $Q_1\supset \langle S_i\rangle$ and $\corank(Q_1)=3$ because the singularities of $X$ are isolated.
Then two pencils of four-dimensional 
linear subspaces 
$\Lambda_t$ and $\Lambda_t'$ on $Q_1$ give us two pencils $|S_i|=\Lambda_t\cap Q_2\cap Q_3$ and $|S_{i+2}|=\Lambda_t'\cap Q_2\cap Q_3$ of quartic del 
Pezzo surfaces. Since $X$ contains two more pencils of quartic del 
Pezzo surfaces, there is another quadric, say $Q_2$ in the net such that $\corank(Q_2)=3$.
The last statement of~\ref{ex:r=3c} is proved similarly to the proof of 
Proposition~\ref{claim:case1-0}.
\end{proof}

\newcommand{\etalchar}[1]{$^{#1}$}
\def\cprime{$'$}


\end{document}